\documentclass[12pt]{article}
\usepackage{amssymb,latexsym,amsmath}

\usepackage{graphics}
\usepackage{pstricks}
\usepackage{pst-node}

\newtheorem{theorem}{Theorem}[section]
\newtheorem{proposition}[theorem]{Proposition}
\newtheorem{lemma}[theorem]{Lemma}
\newtheorem{claim}[theorem]{Claim}
\newtheorem{definition}[theorem]{Definition}
\newtheorem{fact}[theorem]{Fact}
\newtheorem{corollary}[theorem]{Corollary}
\newtheorem{remark}[theorem]{Remark}

\newcommand{\qed}{\relax\ifmmode\eqno\Box\else\mbox{}\quad\nolinebreak\hfill$\Box$\smallskip\fi}

\newcommand{\Zp}{{\mathbb Z}_p}

\newenvironment{proof}{{\it Proof} :}{\qed}

\newcommand{\nc}{\newcommand}
\def\vlabel{\label}

\newcommand{\Kpinf}{K^{p^{\infty}}}

\newcommand{\Gpinf}{p^{\infty} G(K)}

\nc{\K}{K[X_{\infty}]}
\nc{\KKn}{K[X_{1 \leq n},...,X_{k \leq n}]}

\nc{\KKK}{K[X_{1\infty},\ldots,X_{k\infty}]}
\nc{\KxK}{K[x_{1\infty},\ldots,x_{k\infty}]}

\nc{\Iacf}{I_{acf}}
\nc{\scftype}{(scf)-type\  }
\nc{\acftype}{(acf)-type\ }
\nc{\scftypes}{(scf)-types\ }
\nc{\acftypes}{(acf)-types\ }
\nc{\tacf}{t_{acf}}
\font\help=cmr7
\def\strsubset{\hbox{$\subseteq\kern-.8em\lower.2em\hbox{\help
\char'57}$}\,}
\def\dnfo{\,\raise.2em\hbox{$\,\mathrel|\kern-.9em\lower.35em\hbox
{$\smile$}$}}
\def\dnf#1{\lower.9em\hbox{$\buildrel\dnfo\over{  #1}$}}
\def\dfo{\;\raise.2em\hbox{$\mathrel|\kern-.9em\lower.35em\hbox{$\smile$}
\kern-.7em\hbox{\char'57}$}\;}
\def\df#1{\lower.9em\hbox{$\buildrel\dfo\over{ #1}$}}

\nc{\Conj}{\wedge \!\! \wedge} 
\newcommand{\infdef}{ $\Conj$-definable }

\newcommand{\Oc}{\mathcal{O}}

\newcommand{\A}{\mathbb{A}}
\newcommand{\Fp}{\mathbb{F}_p}
\newcommand{\N}{\mathbb{N}}
\newcommand{\Z}{\mathbb{Z}}

\newcommand{\Gm}{\mathbb{G}_m}
\newcommand{\Ga}{\mathbb{G}_a}
\newcommand{\C}{\mathcal{C}}

\title{Semiabelian varieties over separably closed fields, maximal divisible subgroups, and exact sequences}
\author{Franck Benoist\thanks{Supported by a Post Doctoral position of
    the Marie Curie European Network  MRTN CT-2004-512234 (MODNET)
    (Leeds, 2007-2008).} \thanks{Partially supported by ANR MODIG (ANR-09-BLAN-0047) Model theory and Interactions with Geometry}
\\University of Leeds and Univ. Paris-Sud
\and Elisabeth Bouscaren\footnotemark[2]\\CNRS - Univ. Paris-Sud 
\and
Anand
Pillay\thanks{Supported by a Marie Curie Excellence Chair 024052, and EPSRC grants EP/F009712/1 and EP/I002294/1.
}\\University of Leeds and University of Notre-Dame}

\begin{document}

\maketitle
\begin{abstract} Given a separably closed field $K$ of characteristic
  $p>0$ and finite  degree of imperfection  we study the
  $\sharp$-functor which takes a semiabelian variety $G$ over $K$ to 
the maximal divisible subgroup of $G(K)$. Our main result is an example where $G^{\sharp}$, as a ``type-definable group" in 
$K$, does not have ``relative Morley rank", yielding a counterexample to a claim in 
\cite{Hrushovski}. Our methods involve studying the question  of the preservation of exact sequences by the $\sharp$-functor, and relating this to issues of descent as well as model theoretic properties of $G^{\sharp}$. 
We mention some
characteristic $0$ analogues of these ``exactness-descent" results, where differential algebraic methods are more prominent. We also develop the notion of an iterative $D$-structure on a group scheme over an iterative Hasse field, which is interesting in its own right, as well as providing a uniform treatment of the characteristic $0$ and characteristic $p$ cases of ``exactness-descent". 
\end{abstract}

\section{Introduction}\vlabel{Introduction}
 
For a semiabelian variety $G$ over a separably closed field $K$ of characteristic $p>0$ and finite degree of imperfection, the group 
${\Gpinf}$ = $\cap_{n} p^{n}(G(K))$ played a big role in Hrushovski's
proof of the function field Mordell-Lang conjecture in positive characteristic.  The group ${\Gpinf}$ which we also sometimes call $G^{\sharp}$, is {\em type-definable} in the structure 
$(K,+,\cdot)$.  It was claimed in \cite{Hrushovski} (in the Remark just before
lemma 2.15) that ${\Gpinf}$ always has {\em finite relative Morley rank}. 
One of the reasons or motivations for writing the current paper is to
show that this is not the case: there are $G$ such that ${\Gpinf}$ does
not even have  relative Morley rank. (Note that however Lemma 2.15
itself does
hold, the generic type of ${\Gpinf}$ is indeed ``thin'' which implies
that ${\Gpinf}$ {\em does} have finite $U$-rank, but just  not that it has
finite relative Morley rank. The finiteness of $U$-rank suffices for all
the results  in section 4 of \cite{Hrushovski}, in particular
Proposition 4.3, to go through, 
hence the validity of the main results of \cite{Hrushovski} is
unaffected.)
Hrushovski used expressions such as ``Morley dimension" or ``internal Morley dimension" for what we call here relative Morley rank. The notion is somewhat subtle and concerns performing a Cantor-Bendixon analysis {\em inside} a closed space of types. Details and examples are given in section 2.3.

As the second author noticed some time ago, the ``relative Morley rank" problem is related in various ways to whether the 
$p^{\infty}$
(or ${\sharp}$)-functor preserves exact sequences. So another theme of the current paper is
to give conditions on an exact sequence $0\to G_{1} \to G_{2} \to G_{3}\to 0$ of semiabelian varieties over $K$ which imply exactness of the sequence $0 \to G_{1}^{\sharp} \to G_{2}^{\sharp} \to G_{3}^{\sharp} \to 0$, as well as giving situations where the sequence of $G_{i}^{\sharp}$ is NOT exact.

A third theme relates the preservation of exactness by $\sharp$ to the issue of descent of a semiabelian variety $G$ over $K$ to the field of ``constants" ${\Kpinf}$  =  $\cap_{n}K^{p^{n}}$ of $K$.

If $K$ has degree of imperfection $e$ (meaning that $K$ has dimension $p^{e}$ as a vector space over its $pth$ powers $K^{p}$), then $K$ can be equipped naturally with $e$ commuting  iterative Hasse derivations. We will, for simplicity, mainly consider the case where $e=1$  (so for example where $K = {\Fp}(t)^{sep}$), in which case we have a single iterative Hasse derivation $(\partial_{n})_{n}$  whose field of absolute constants is ${\Kpinf}$. 
This differential structure on $K$ will play a role in some proofs, by virtue of so-called D-structures on varieties over $K$. 

The analogue in characteristic $0$ of the differential field $(K,(\partial_{n})_{n})$ is simply a differentially closed field $(K,\partial)$ (of characteristic zero).  And for an abelian variety $G$ over our characteristic 
$0$ differentially closed field $K$  we have what is often called the ``Manin kernel" for $G$, the smallest Zariski-dense ``differential algebraic" subgroup of $G(K)$, which we denote again by $G^{\sharp}$. The issues of preservation of exactness by $\sharp$ and the relationship to descent to the field ${\cal C}$ of constants, make sense in characteristic $0$ too.

In characteristic $p$, it is possible to obtain our results with a purely algebraic approach using $p$-torsion and Tate modules (carried out in section \ref{charp}). In characteristic $0$, we need to use differential algebraic methods, in particular D-structures. But in fact 
the algebraic proofs given in characteristic $p$ can also be seen as involving D-structures and we take the opportunity of giving such a uniform proof in all characteristics in section \ref{sectionuniform}.

Our paper builds on earlier work by the second author and 
Fran\c coise Delon \cite{BD2} where among other things, the groups $G^{\sharp}$ (in positive characteristic) are characterized as precisely the commutative divisible type-definable groups in separably closed fields.
Our results, especially in characteristic $0$, are also influenced by and closely related to themes in the third author's joint paper with Daniel Bertrand \cite{BePi}.

Let us now describe the content and results of the paper.


Section 2 recalls key notions and facts about differential fields, and semiabelian varieties over separably closed fields. We also discuss relative Morley rank, preservation of descent under isogeny,  and some properties of $p^{\infty}G(K)$.       

In section 3 we introduce the $\sharp$-functor in all characteristics and begin relating relative Morley rank to exactness. 

Section \ref{charp} concentrates on the characteristic $p$ case. We begin by making some observations about  descent of semiabelian
varieties and Tate modules,  proving for
example that an ordinary semiabelian variety $G$
descends to the constants of $K$ if and only if all of the (power of $p$)-torsion of $G$ is $K$-rational (section \ref{TorsionTate}).
We make no claim that our results on descent are especially novel, and we would not be surprised if they were explicit or implicit in the literature on semiabelian varieties in positive characteristic. 
However we were unable to find precise references in spite of consulting several experts. 
In section \ref{exactnessdescent}, we answer the original question which motivated this paper. In Proposition \ref{Maindescentcarp}, we show that if $0\to G_1 \to G_2 \to G_3 \to 0$ is an
exact sequence of ordinary semiabelian varieties such that $G_{1}$ and $G_{3}$ descend to the constants, ${\cal C}$, then the sequence of $G_{i}^{\sharp}$'s is exact if and only if $G_{2}$ descends to ${\cal C}$.
This yields an example of a semiabelian variety $G$ such that $G^{\sharp}$ does not have relative Morley rank (in fact the example is simply any nonconstant extension of a constant ordinary abelian variety by an algebraic torus). See Corollary 4.14, which as mentioned above is among the main results of our paper.
The remainder of section \ref{charp} contains both positive and negative results about preservation of exactness by $\sharp$ in various situations. In particular we give an example of an exact sequence of ordinary abelian varieties for
which the $\sharp$-functor does not preserve exactness. This cannot happen in characteristic $0$ as shown in the next section.

In section \ref{sectionuniform}, we switch to differential algebraic methods in order to treat uniformly both characteristic $0$ and characteristic $p$. In section \ref{sectionDstructures} 
we recall the definition of D-structures for group schemes and the fact that a semiabelian
variety $G$ over a Hasse field $K$ descends to the constants of $K$ if and only if $G$ admits an iterative D-structure. 

In order to relate exactness of the $\sharp$-functor and descent in characteristic $0$, we use, as in \cite{BePi}, the universal extension $\tilde G$ of $G$ by a vector group, which always
admits a unique D-structure. In characteristic $p$, we need to replace this universal extension by a ($p$-divisible) proalgebraic group, also called $\tilde G$. In section \ref{Gtildecharp}, in characteristic $p$, we endow $\tilde G$ with an iterative D-structure
and prove the characteristic $p$ version of the characteristic $0$ results relating descent and the D-structure on $\tilde G$. Finally, in section \ref{unifstat}, we can then give a uniform proof, in all characteristics (Proposition \ref{Maindescent}), 
of the fundamental result (Proposition 
\ref{Maindescentcarp})
proved previously in 
characteristic $p$.



We should say that, as far as ``algebraic geometry" is concerned this paper is elementary, and, even in section \ref{unifstat} does not make heavy use of modern methods.  The reader is referred to \cite{Conrad} for a modern scheme-theoretic treatment of descent, $K/k$-trace etc.,  for abelian varieties in positive characteristic. As is pointed out there, much of the literature on such questions and on important results such as the Lang-Neron theorem, remains in the language of  Weil. The same will be to some extent true of the current paper, where our real aim and motivation is to understand $G(K)$ as a definable group in the structure
$(K,+,\cdot)$, as well as its type-definable subgroups.


Elisabeth Bouscaren 
would like to thank particularly  Ehud Hrushovski and
 Fran\c coise Delon for numerous discussions in the past years on  the
 questions addressed in this paper. Grateful thanks  from all 
 three authors go especially  to Daniel Bertrand and Damian R\"ossler for numerous
 and enlightening discussions. Among the many others who have helped with
 explanations or discussions with some of the authors, let us
 give special thanks to Jean-Benoit Bost, Antoine Chambert-Loir, Marc
 Hindry, Minhyong Kim and Thomas Scanlon. Finally thanks to the referees
 of earlier versions  for their  comments, and   particular thanks to the
 final  referee for his/her careful reading and numerous extremely
 helpful remarks and  suggestions.

\section{Preliminaries}

\subsection{Hasse fields}\vlabel{Hasseprelim} 
We summarise here basic facts and notation about the fields $K$ that
concern us. More details can be found in \cite{BeDe},
\cite{Ziegler1} for the characteristic $p$ case  and \cite{marker} for
the characteristic zero case. 

\vspace{2mm}
\noindent
If $K$ is a separably closed field of characteristic $p >0$ then the dimension of $K$ as a vector
space over the field $K^{p}$ of $p^{th}$ powers is infinite or a power
$p^{e}$ of $p$. In  the second case, $e$  is called the degree of
imperfection (we will just say the ``invariant")  of $K$ and we will be interested in the case when $e \geq 1 $ (and often when $e=1$).  For $e$ finite, a $p$-basis of $K$ is a set $a_{1},..,a_{e}$ of elements of $K$ such that 
$\{a_{1}^{n_{1}}a_{2}^{n_{2}}...a_{e}^{n_{e}}: 0\leq n_{i} < p^e\}$ form a basis of $K$ over $K^{p}$.

\vspace{2mm}
\noindent
The first order theory of separably closed fields of characteristic $p>0$ and invariant $e$ (in the language of rings) is complete (and model complete). We call the theory $SCF_{p,e}$. It is also stable (but not superstable) and certain natural (inessential) expansions that we mention below, have quantifier elimination.

\vspace{2mm}
\noindent
For $R$ an arbitrary ring (commutative with a $1$), an {\em iterative Hasse derivation} $\partial$ on $R$ is a sequence $(\partial_{n}:n = 0,1,...)$ of additive maps from $R$ to $R$ such that
\newline
(i) $\partial_{0}$ is the identity,
\newline
(ii) for each $n$, $\partial_{n}(xy) = \sum_{i+j = n}\partial_{i}(x)\partial_{j}(y)$, and
\newline
(iii) for all $i,j$,  $\partial_{i}\circ\partial_{j} = {i+j\choose i}\partial_{i+j}$ (iterativity).

\vspace{2mm}
\noindent
Note that $\partial_{1}$ is a derivation, and that when $R$ has characteristic $0$, $\partial_{n} = \partial_1^{n}/n!$ (So in
the characteristic $0$ case the whole sequence $(\partial_{n})_{n}$ is determined by $\partial_{1}$.)\\
In some rare cases we will speak about non iterative Hasse derivation, meaning that the third condition is not required.

\vspace{2mm}
\noindent
 By the {\em constants} of
  $(R,(\partial_{n})_{n\geq 0})$ one usually  means $\{r\in R:\partial_{1}(r) = 0\}$ and by the 
{\em absolute constants} $\{r\in R: \partial_{n}(r) = 0$ for all
$n>0\}$. In this paper, we will mainly consider the field of absolute
constants, denoted $\C$, and refer to them  in the sequel as ``the constants''.

\vspace{2mm}
\noindent
If $\partial^{1}$ and $\partial^{2}$ are iterative Hasse derivations on $R$ we say that they commute if each $\partial^{1}_{i}$ commutes with each $\partial^{2}_{j}$. 

\begin{fact} (i) If $K$ is a separably closed field of invariant $e \geq
  1$, then there are commuting  iterative Hasse derivations $\partial^{1},..,\partial^{e}$ on $K$ such that the common constants of $\partial_1^{1},..,\partial_1^{e}$ is $K^{p}$. In this case the common (absolute) constants of $\partial^{1},..,\partial^{e}$ is the field $K^{p^{\infty}} = \cap_{n}K^{p^{n}}$. 
\newline 
(ii) Moreover in (i), if $a_{1},..,a_{e}$ is  a $p$-basis of $K$, then each $\partial^{i}_{j}$ is definable
in the field $K$ over parameters consisting of the $a_{1},..,a_{e}$ and their images under the maps $\partial^{n}_{m}$  ($n = 1,..,e$, $m \geq 0$). 
\newline
(iii) The theory $CHF_{p,e}$ of separably closed fields of degree $e$,
equipped with $e$ commuting iterative Hasse derivations
$\partial^{1},..,\partial^{e}$, whose common field of constants is
$K^{p}$, is complete, stable,  with quantifier elimination (in the language of rings
together with unary function symbols for each $\partial^{i}_{n}$, $i=1,..,e$, $n > 0$). 
\end{fact}

Note that after adding names for a $p$-basis $a_{1},..,a_{e}$ of the separably closed field $K$, we obtain for each $n$ a basis $1,d_{1},..,d_{p^{n}-1}$ of $K$ over $K^{p^{n}}$, and the functions $\lambda_{n,i}$ such that 
$x = \sum_{i} (\lambda_{n,i}(x))^{p^{n}}d_{i}$ for all $x$ in $K$, are definable with parameters $a_{1},..,a_{e}$ in the field $K$. The theory of separably closed fields also has quantifier elimination in the language with symbols for a $p$-basis and for each $\lambda_{n,i}$.  The relation between the $\lambda$-functions and the $\partial^{i}_{j}$ is given in section 2 of \cite{BeDe}.

In the current paper we concentrate on the iterative Hasse derivation
formalism. In fact when we mention separably closed fields $K$ with an
iterative Hasse structure, we will usually assume that $e = 1$ and so 
$K$ is equipped with a single iterative Hasse derivation $\partial =
(\partial_{n})_{n}$. The basic example is ${\mathbb F}_{p}(t)^{sep}$
(where $^{sep}$ denotes separable closure) with  $\partial_{1}(t)
= 1$ and $\partial_{i}(t) = 0$ for all $i>1$.
The assumption that $e =1$ is made here for the sake of simplicty, as
some of the results we will be quoting are only explicitly written out
for this case, but it will be  no real restriction, thanks to: 

\begin{fact}\vlabel{invariant1} (see for example \cite{BeDe}) Let $K_0$ be an algebraically
  closed field of characteristic $p$, and $K_1$ a finitely generated
  extension of $K_0$. Then there is a separably
  closed field $K$ of degree of imperfection $1$, extending $K_1$ and
  such that $K_0 = K^{p^\infty}$.\end{fact}

\vspace{2mm}
\noindent
Our characteristic $0$ analogue is simply a differentially closed field
$(K,\partial)$ of characteristic $0$, where now $\partial$ is the single
distinguished derivation (rather than a sequence). The corresponding
first order theory is $DCF_{0}$, in the language of rings together with
a symbol for $\partial$. The theory $DCF_{0}$ is complete with quantifier elimination, but is now $\omega$-stable. 

\medskip

\subsection{Varieties, semiabelian varieties and separable morphisms}\vlabel{prelimcharp}

From now on, $K$ is an algebraically closed field of characteristic $0$, or a separably closed field of
characteristic $p$ and of finite degree of imperfection $e\ge 1$, and $\overline K$ denotes an
algebraic closure of $K$.


As already mentioned in the introduction, we will use mainly Weil type
language in this paper, except in section \ref{Gtildecharp}. A {\em variety over $K$}, or {\em defined over
  $K$}, will always be a separated reduced scheme of finite type over
$K$.  We denote by $V(L)$ the set of $L$-rational points of $V$, for $L$ an extension of $K$. Recall
that when $K$ is separably closed,  and $V$ is over $K$, $V(K)$ is
Zariski dense in $V$. We will often identify $V$ with its set of geometric points $V(\overline K)$.
For $L$ an extension of $K$, we will denote $V_L=V \times_K L$ (extension of scalars or base change).


Recall that if $V$ and $W$ are two irreducible varieties over $K$, and
$f$ is a dominant $K$-morphism from $V$ to $W$, $f$ is said to be {\em separable}
if the function field extension $K(W)\subset  K(V)$ is separable.\\
The following is classical. For the convenience of the reader, we
include a short (model-theoretic)  proof in the Appendix A. 


\begin{fact} \vlabel{separablemorphisms} Let $G, H$ be two connected 
algebraic  groups defined over $K$ and $f$ a
dominant separable homomorphism from $G$ to $H$ (equivalently a
surjective separable homomorphism from  $G(\overline K)$ onto
$H(\overline K)$). Then $f$ takes $G(K)$ surjectively onto 
$H(K)$. \end{fact}

{\em In this paper we will only consider  exact sequence of algebraic groups  
$$ 0 \rightarrow G_1 \stackrel{g}{\rightarrow}
  G_2\stackrel{f}{\rightarrow} G_3\rightarrow 0$$  such that both
  morphisms are separable}. These are sometimes called strict exact
sequences (\cite{Serrebook}). We will say also that $G_2$ is an algebraic group extension
  of $G_3$ by $G_1$, denoted by  $G_2 \in EXT(G_3,G_1)$. By the assumption of
  separability of the morphisms,   $G_3$ is isomorphic (as an algebraic group) to
  $G_2/g(G_1)$ and $G_1$ is isomorphic to a closed subgroup of $G_2$.   

We will say that the   exact sequence  {\em is  over $K$}    if  the groups $G_1,G_2,
  G_3$ are algebraic groups over $K$ and   $f,g$ are separable $K$-morphisms
  of algebraic groups.



We now  recall some very basic facts about semiabelian varieties.
 We will be particularly interested in rationality issues,
that is in the groups of $K$-rational points of some basic subgroups of $G(K)$. There
are many classical references for  abelian varieties (for
example \cite{Mumford}, or \cite{Lang}).  For the case of tori, see for
example \cite{Borel}.  

It is then easy to obtain the corresponding facts for the case of
arbitrary semiabelian varieties.

Recall that a {\em semiabelian}  variety $G$ (over $K$) is an extension of an abelian 
variety by a torus, i.e.  
$$0 \rightarrow T \rightarrow
  G\rightarrow  A\rightarrow 0$$
where  $T$ is a torus over $K$, $A$ is an abelian  variety
  over $K$ and the two morphisms
  are separable $K$-morphisms ($G$ is then also an algebraic group over $K$).

The following facts hold when $K$ is separably closed: 

\begin{fact} \vlabel{subgroups}(i) Let $T$ be a torus over $K$. Then $T$ is
  $K$-split, that is $T$ is isomorphic {\em over $K$} to some product of the
  multiplicative group,  $({\Gm})^{\times n}$. Any closed
  subgroup of $T_{\overline K}$ is then also defined over $K$.\\
(ii) Semiabelian varieties are commutative and divisible,
  i.e. $G(\overline K ) $, the group of $\overline K$-rational points
  of $G$ is a commutative divisible group. \\
(iii) Let $G$ be a semiabelian variety over $K$, then any closed
  connected subgroup of $G_{\overline K}$ is defined over $K$. 
\end{fact}

\begin{definition}
Let $K_0 \subset K_1$ be an extension of fields, and $G$ an algebraic group over $K_1$. We will say that 
$G$ descends to $K_0$ if $G$ is isomorphic to $H_{K_1}$ for some algebraic group $H$ over $K_0$.
\end{definition}


As semiabelian varieties are defined as extensions, one should check
what descent exactly means in that case. The following fact, 
which follows from classical manipulations on $EXT(A,T)$ (see for example
\cite{Serrebook}), deals with this question.

\begin{fact} \vlabel{descsemiab}
Let $K_0\subset K_1$ be separably closed fields, and $G$ a semiabelian variety defined over $K_1$, which is an extension of $A$ by $T=(\Gm^n)_{K_1}$. If $G$ descends to $K_0$, i.e. if $G$ is isomorphic to $(G_0)_{K_1}$ for some
semiabelian variety $G_0$ over $K_0$, then we have the following
~\\
\begin{center}
\begin{pspicture}(7,2)
\rput(0,2){\rnode{A}{$0$}}
\rput(1,2){\rnode{B}{$T$}}
\rput(3,2){\rnode{C}{$G$}}
\rput(5,2){\rnode{D}{$A$}}
\rput(7,2){\rnode{E}{$0$}}
\rput(0,0){\rnode{A0}{$0$}}
\rput(1,0){\rnode{B0}{$T$}}
\rput(3,0){\rnode{C0}{$(G_0)_{K_1}$}}
\rput(5,0){\rnode{D0}{$(A_0)_{K_1}$}}
\rput(7,0){\rnode{E0}{$0$}}
\psset{arrows=->,nodesep=3pt,shortput=tablr,linewidth=0.1pt}
\ncline{A}{B}
\ncline{B}{C}^{$i$} 
\ncline{C}{D}^{$f$}
\ncline{D}{E}
\ncline{A0}{B0}
\ncline{B0}{C0}^{$i_0$} 
\ncline{C0}{D0}^{$f_0$}
\ncline{D0}{E0}
\ncline{B}{B0}>{$id$}
\ncline{C}{C0}>{$g$}
\ncline{D}{D0}>{$h$}
\end{pspicture}
\end{center}
~\\
where $g$ and $h$ are isomorphisms, and $0 \to (\Gm^n)_{K_0} \to G_0 \to A_0 \to 0$ is a semiabelian variety over $K_0$. Furthermore, if $A$ is of the form $(A_0)_{K_1}$ for some $A_0$ over $K_0$, we can choose
$h$ to be the identity.
\end{fact}

\begin{proposition} \vlabel{modulispace} Assume $char(K)=p$. Let $G$ be a semiabelian variety over $K$, such that $G$ descends to $K^{p^n}$ for all $n\ge 0$. Then $G$ descends to $K^{p^{\infty}}$.
\end{proposition}

\begin{proof}
Let $A$ be the abelian part of $G$, and $T$ its toric part. 
By Fact \ref{descsemiab}, $A$ descends to $K^{p^n}$ for all $n$. Using a suitable moduli space (namely the moduli space of abelian varieties equipped with a polarization of fixed degree and an $m$-level structure,
see \cite{MumfordFo}), it follows that $A$ descends to $K^{p^{\infty}}$.\\
Now fix $A_0$ over $K^{p^{\infty}}$ such that $A \cong (A_0)_K$. 
It is classical that  $\text{Ext}(A,T)\simeq
(\text{Ext}(A,\Gm))^t\simeq (\hat A)^t$, where $\hat A$ is the dual abelian variety of $A$, and is isomorphic to $(\hat A_0)_K$ (see 
for example \cite{Serre}). Using Fact \ref{descsemiab} again, and since $G$ descends to $K^{p^n}$ for each $n$, the isomorphism type of $G$
is parametrized by a point in 
$\hat A_0(\bigcap_n K^{p^n})=\hat A_0(K^{p^{\infty}})$, that is, $G$ descends to $K^{p^{\infty}}$.
\end{proof}


\begin{remark} Over a  separably closed field $K$  of  characteristic $p>0$, the
  semiabelian varieties over $K$ are exactly the commutative 
divisible algebraic groups over $K$. Indeed let $H$ be commutative
divisible,   consider the
biggest connected affine subgroup of $H$, $T$.  By divisibility it must be
 a torus and as $K$ is
separably closed, it is defined over $K$  (and split over $K$), and
$H/T$ is an abelian variety, by Chevalley's theorem (\cite{Rosenlicht1}).
\end{remark}
\subsubsection{Torsion}

The behaviour of the torsion elements of $G$ is particularly
important in characteristic $p$. The following classical facts will enable us to fix some notation
for the rest of the paper. 

\begin{fact} \vlabel{torsion} Let $G$ be a semiabelian variety over $K$, written additively, and $$0 \rightarrow T \rightarrow
  G\rightarrow  A\rightarrow 0,$$  with $dim(A) =a$ and $dim (T) =
  t$\\
{\em 1.} If $n$ is prime to $p=char(K)$ or arbitrary in characteristic $0$, then 
$[n] : G \mapsto G$, $x \mapsto nx$  is a separable isogeny of degree (= separable
  degree) $n^{2a + t}$. We denote by $G[n]$ the kernel of $[n]$, the
  points of $n$-torsion, then  $G[n](\overline K) \cong
  ({\mathbb Z}/ n {\mathbb Z})^{2a +t}$. By separability,  $G[n](\overline K) = G[n](K)$.\\
From now on, $char(K)=p>0$.\\
{\em 2. } $[p]:  G \mapsto G$ is an inseparable isogeny of degree $p^{2a +t}$, and of
  inseparable degree at least $p^{a+t}$. Hence there is some $r$, $0 \leq r
  \leq a$ such that, for every $n$,  
$$ G[p^n](\overline K)=Ker [p^n](\overline K) \cong  ({\mathbb Z}/ {p^n} {\mathbb Z})^{r}.$$
We say that $G$ is {\em ordinary} if $r = a$ (note that tori are ordinary semiabelian varieties). \\
As $G[p^n](\overline K)$ is finite, it is contained in $G(\overline K)$, but not
necessarily in $G(K)$. \\ 
{\em 3. } Let $G[p^\infty] $ or $G[p^\infty](\overline K)$  denote the elements of $G$ with order a power of $p$,
and $G[p']$ or $G[p'](\overline K)$ denote the elements of $G$ with
order prime to $p$.  
Then $G[p'] =  G[p'](K)  $ is Zariski dense in $G$.\\
Note that,  even for
 $G$ ordinary, we may well have that $G[p^\infty](K) = \{0\}$. 
\end{fact}

We will also need the following easy observations:


\begin{fact}\vlabel{torsionandexactsequences} Let $ 0 \rightarrow G_1 \rightarrow
  G_2\stackrel{f}{\rightarrow} G_3\rightarrow 0$ be  an exact  sequence of
  semiabelian varieties over $K$. Then \\
for every $n$, the restriction of $f$ to $n$ torsion induces an exact sequence (in the category of groups),
i.e. $$ 0 \rightarrow G_1[n](\overline K) \rightarrow
  G_2[n](\overline K) \stackrel{f}{\rightarrow} G_3[n](\overline K)
  \rightarrow 0. $$ 
It follows in particular that, in all characteristics, 
$$ 0 \rightarrow Tor G_1  \rightarrow
  Tor G_2 \stackrel{f}{\rightarrow} Tor G_3\rightarrow 0 $$
is an exact sequence of groups, 
  where $Tor G$ denotes the group of all torsion elements of
  $G(\overline K)$.
\end{fact}


\medskip 

Divisibility  by $p$ also behaves quite differently in $G(\overline K)$ and in
$G(K)$ when $char(K)=p$. Let 
$$p^\infty G(K) := \bigcap_{n\geq 1} [p^n] G(K).$$

\begin{proposition} \vlabel{divisibility} 1. $G(K)$ is $n$-divisible for any $n$ prime to $p$. \\
2. For $n$ prime to $p$, for every $k$, $G[n](\overline K) =G[n](K)\subset [p^k] G(K)$.\\
3. $G[p'](\overline K)=G[p'](K)$ is a divisible subgroup of $G(K)$.\\
4. $p^\infty G(K)$ is $n$-divisible for any $n$ prime to $p$.\\
5. $p^\infty G(K)$ is infinite and Zariski dense in $G$. \\
6. $p^\infty G(K)$ is the biggest divisible subgroup of $G(K)$.
\end{proposition}
\begin{proof} 1 to 5 are clear from previous facts.\\
6 follows from K\"onig's Lemma and the finiteness of $G[p^n]$ for every $n$.
\end{proof}

\subsubsection{Isogenies and descent in char.p}\vlabel{isogenies}

We will not necessarily directly use all  the classical facts about isogenies
 recalled below, but they give a picture of the various
problems linked to descent questions in characteristic $p$. We will provide short elementary proofs
when they exist.

In this section, $K$ is any separably closed field of characteristic
$p>0$, $G$ and $H$ are semiabelian varieties over $K$.

Note first that if $G$ and $H$ are semiabelian varieties over $K$, and $f$ a morphism of algebraic groups
$G_L \to H_L$ for some extension $L\supset K$, then $f$ is actually defined over $K$, i.e. $f=g_L$ for some
$K$-morphism $g$ from $G$ to $H$:  
by \ref{subgroups},
the graph of $f$,  which is a closed connected subgroup of $(G \times H)_L$
is also defined over $K$.  

Recall   that an {\em isogeny} is a surjective morphism of
algebraic groups with finite kernel. 


Let  $G$ be  a semiabelian variety over $K$. It is classical that for 
every 
$n\geq 1$ the relative $n^{th}$-Frobenius   isogeny $F^n : G \longrightarrow G^{(p^n)}$
($G^{(p^n)}$ descends to $K^{p^n}$)  is purely inseparable of degree
$p^{n dim G}$, and 
admits a ``quasi-inverse'' isogeny, the $n^{th}$-Verschiebung, denoted $V_n : G^{(p^n)}
\longrightarrow G $, such that $V_n \circ F^n = [p^n]_G $ and $F^n
\circ V_n = [p^n]_{ G^{(p^n)}}$. It is easily seen, counting degrees, that:

\begin{fact}\vlabel{separableverschiebung} If $G$ is ordinary, then for
  every $n$, the Verschiebung $V_n$ is separable. \end{fact}

\begin{lemma} \vlabel{divweil2} Let $G$ be a semiabelian variety over
 $K$ and $L$ an extension of $K$. Then if $a \in p^n G(L)$, there exists $b \in G(L)$ such that $a\in
 K(F^n(b))$. So if $G$ is over $K^{p^n}$, then $[p^n] G(K)
 \subset  G(K^{p^n})$ and in particular  $\Gpinf
 = p^{\infty} G(K^{p^n})$.\end{lemma}
\begin{proof} Consider the $n^{th}$-Verschiebung $ V_n$, described
 above. 
If $a \in p^n G(L)$, then $a = p^n b $ for some $b \in
 G(L)$, and $a = V_n (F^n (b))$. If $G$ is over $K^{p^n}$, then
 the Verschiebung is also over $K^{p^n}$ and  $a
 \in K^{p^n}(F^n (b)) = K^{p^n}$. \end{proof}




Abelian varieties have one specific very important property: 

\begin{fact}\vlabel{abeliansimple} Let $A$ be an abelian variety over $K$. Then $A$
  is isogenous over $K$ to a finite product of simple (i.e. which have no
  proper nontrivial closed connected subgroup) abelian varieties.\end{fact}

We will now recall some very classical facts about descent. We will try
to point out where the difficulties arise, for our readers not
completely familiar with the theory of abelian varieties in
characteristic $p$.  

In characteristic $0$, any semiabelian variety which is isogenous to one defined
over some algebraically closed $K_0$ descends, in the sense above, to
$K_0$ (i.e. Fact \ref{easyseparableisogeny} applies). The
situation is more complicated in characteristic $p$.

 \begin{fact} \vlabel{easyseparableisogeny} Let $K_0\subset K_1$, with $K_0$ algebraically closed. Let $G$
  be a semiabelian  variety over $K_0$, $H$ a semiabelian variety 
  over $K_1$ and $f$ a separable isogeny from $G_{K_1}$ onto
  $H$.  Then $H$ descends to $K_0$. \end{fact} 
\begin{proof} As $f$ is a separable isogeny, the kernel of $f$ is a
  finite closed subgroup of $G(K_0)$, $N$, of cardinality the degree (=
  separable degree)  of
  $f$. Then $G':= G/N$ is a semiabelian variety over $K_0$, and
  $f$ induces an isomorphism from $H$ onto $G'_{K_1}$.\end{proof}

\smallskip 

The following is  also classical, but more complicated and {\em is only true
for abelian varieties}. 

\begin{fact} \vlabel{hardseparableisogeny}  Let $K_0\subset K_1$, with $K_0$ algebraically closed. Let $A$
  be an abelian variety over $K_1$, $B$ an abelian variety 
  over $K_0$ and $f$ a separable isogeny from  $A$ onto
  $B_{K_1}$. Then $A$ descends to $K_0$.\end{fact}
\begin{proof}
This is a particularly simple case of the ``Proper base change theorem''
(see for example in \cite{SGA1} or \cite{Milne}). \end{proof}

\smallskip

\begin{remark} \vlabel{ellipticdescent} 
Note that in the case of dimension one, Fact \ref{easyseparableisogeny} holds without the assumption that $f$
is separable. That follows easily from the fact that in dimension one, an isogeny factors through some power of the Frobenius 
 (see for example \cite{silverman}).\\
We will give later (Remark \ref{supersingular}) an example showing that Facts \ref{easyseparableisogeny} and \ref{hardseparableisogeny} do not hold without the 
separability assumption in dimension $>1$.
\end{remark}


\bigskip

\subsection{Relative Morley Rank}\vlabel{Morleyrank}
In this section $T$ will be a complete theory, and we work in a given 
$\kappa$-saturated model $M$, for $\kappa$ sufficiently big.  
We will here define {\em relative} Morley rank, namely Morley rank inside a given {\infdef} set. This was called 
{\em internal Morley dimension} in \cite{Hrushovski}. By an {\infdef}
set (infinitely definable set) we mean a subset of some $M^{n}$ which is the intersection of a small (size $< \kappa$) collection of definable subsets of $M^{n}$ (that is the set of realizations of a partial type over a small set of parameters). We will fix an {\infdef} set $X\subseteq M^{n}$.

If $X$ is an infinitely definable subset of $M^n$, by a {\em relatively definable} subset of $X$ we mean a subset of the
form $Z = X\cap Y$ 
for $Y\subseteq M^{n}$ definable with parameters.  Then we  define Morley rank for relatively definable subsets $Z$ of $X$, as follows: 
\newline
(i) $RM_{X}(Z) \geq 0$ if $Z$ is nonempty.
\newline
(ii) $RM_{X}(Z) \geq \alpha + 1$ if there are $Z_{i}\subseteq Z$ for $i < \omega$ which are relatively definable subsets of $X$, such that $Z_{i}\cap Z_{j} = \emptyset$ for $i \neq j$ and $RM_{X}(Z_{i}) \geq \alpha$ for all $i$.
\newline
(iii) for limit ordinal $\alpha$, $RM_{X}(Z) \geq \alpha$ if  
$RM_{X}(Z) \geq \delta $ for all $\delta < \alpha$. 

\vspace{2mm}
\noindent
As in the absolute case we obtain (relative) Morley degree. Namely suppose that $RM_{X}(Z) = \alpha < \infty$. Then there is a greatest positive natural number $d$ such that $Z$ can be partitioned into $d$ (relatively in X) definable sets $Z_{i}$ such that $RM_{X}(Z_{i}) = \alpha$ for all $i$.

\vspace{2mm}
\noindent
We will say that $X$ has relative Morley rank if $RM_{X}(X) < \infty$.

\vspace{2mm}
\noindent 

\begin{remark} (i) Suppose that $Y$ is a relatively definable subset of $X$. Then $RM_{X}(Y) = RM_{Y}(Y)$.
\newline
(ii) We can also talk about the relative Morley rank $RM_{X}(p)$ of a complete type $p$ of an element of $X$ over a set of parameters. It will just be the infimum of the relative Morley ranks of the (relatively) definable subsets of $X$ which are in p.
\newline
(iii) Suppose that $T$ is countable and $X$ is {\infdef} over a countable set of parameters $A_{0}$. Then $X$  has relative Morley rank if and only if for any countable set of parameters $A\supseteq A_{0}$ there are only countably many complete types over $A$ extending $X$. 
\end{remark}

Now suppose that $X,Y$ are {\infdef} sets and $f: X \to Y$ is a surjective definable function. By definability of $f$ we mean that
$f$ is the restriction to $X$ of some definable function on a definable superset of $X$. Note that then each fibre $f^{-1}(c)$ of $f$ is a relatively definable subset of $X$, so we can talk about its relative Morley rank (with respect to $X$ or to itself, which will be the same by Remark 2.18 (i)).


\begin{lemma} \vlabel{mapRMR}
Suppose $X,Y$ are {\infdef} sets and $f:X\to Y$ is surjective and 
definable. 
\newline 
(i) Suppose that 
$RM_{Y}(Y) = \beta$ and for each $c\in Y$, $RM_{X}(f^{-1}(c)) \leq \alpha$.  Then $RM_{X}(X)\leq \alpha(\beta + 1)$
if $\alpha > 0$, and $\leq \beta$ if $\alpha = 0$.
\newline
(ii) $RM_{Y}(Y) \leq RM_{X}(X)$. 
\end{lemma}
\begin{proof} (i) This is proved in the definable (absolute) case by Shelah \cite{Shelah} (Chapter V, Theorem 7.8) and Erimbetov \cite{Erimbetov}.  Martin Ziegler \cite{Ziegler} also gives a self-contained proof. We point out briefly how Ziegler's proof (see section 2 of \cite{Ziegler}) adapts to the more general context.  
\newline
Case 1, when $\alpha = 0$, \cite{Ziegler} works word-for-word.
\newline
Case 2, when $\alpha > 0$. Work by induction on $\beta$. We may assume that  $Y$ has ``relative Morley degree" $1$ (with respect to itself). 
Suppose for a contradiction that $\alpha\beta + \alpha < RM_{X}(X)$. Lemma 3 of \cite{Ziegler} applies, yielding a relatively definable subset $X'$ of $X$, such that $\alpha\beta < RM_{X}(X')$ and such that the ``generic fibre" of $f|X'$ has finitely many, say $k$, elements (where maybe $k=0$). 
We now apply compactness to find a relatively definable subset $Y^{*}$ of $Y$ such that for all $b\in Y^{*}$, $f^{-1}(b)\cap X'$ has at most $k$ elements. Let $Y' = Y\setminus Y^{*}$ a relatively definable subset of $Y$ such that $RM_{Y}(Y') = \beta' < \beta$.  By Case 1, $X'\cap f^{-1}(Y^{*})$ has relative Morley rank $\leq \beta$, whereby the relative Morley rank of $X'' = X'\cap f^{-1}(Y')$ is $> \alpha\beta \geq \alpha(\beta' + 1)$. This contradicts the induction hypothesis applied to $f|X'': X'' \to Y'$.
\newline 
(ii) is easier, and has the same inductive proof as in the definable (absolute) case, bearing in mind that because $f$ is the restriction to $X$ of a definable function on a definable superset of $X$, the preimage under $f$ of any relatively definable subset of $Y$ is a relatively definable subset of $X$.
\end{proof}

\vspace{5mm}
\noindent

  
If $X = G$ is an {\infdef} group with relative Morley rank then some of the  general theory of totally transcendental groups applies (as already mentioned inside Definition 4.0 of \cite{Hrushovski}). For example we have the $DCC$ on relatively definable subgroups, yielding that $G$ is connected-by-finite among other things. And this is really all we will be using about groups of finite relative Morley rank. 


\vspace{2mm}
\noindent

We now consider an exact sequence of {\infdef} groups 
$1 \rightarrow G_{1}  \stackrel{h}\rightarrow G_{2} \rightarrow G_{3}
\rightarrow 1$. 
We can assume that $G_{1} = Ker(h) \subseteq G_{2}$, as relative Morley rank is preserved by definable bijection, and note again that $G_{1}$ is then a 
relatively definable (normal) subgroup of $G_{2}$. With this notation we
have the following Corollary, which follows immediately from Lemma \ref{mapRMR}:

\begin{corollary} \vlabel{RMRexact} Suppose $G_{1}$ and $G_{3}$ have (finite) relative Morley rank. Then so does $G_{2}$.
\end{corollary}

We complete this section with some additional comments and examples.
\newline
Firstly we obtain the usual (absolute) Morley rank of a definable set $Z\subseteq M^{n}$ by taking $X$ to be  $M^{n}$ in the definition at the beginning of this subsection.
Of course Morley rank can be defined directly for complete types  (over a saturated enough model M), by: $RM(p(x)) = \alpha$ if $p(x)$ is isolated in the subspace of $S_{x}(M)$ obtained by removing the set of types of Morley rank $<\alpha$. Here the ambient space of types is $S_{x}(M)$.  We can make the analogous definition for relative Morley rank $RM_{X}(p)$, by working in the space $S_{X}(M)$ of complete types over $M$ extending the type-definable set $X$.  In any case it should be clear to the reader that $RM_{X}(p)$ need not coincide with $RM(p)$. For example, suppose $X$ is a so-called minimal type-definable set: namely $X$ is infinite and every relatively definable subset of $X$ is finite or cofinite (in $X$). Then there is a unique nonalgebraic complete type over $M$ extending $X$, say $p(x)$. Moreover $RM_{X}(p) = 1$. But $RM(p)$ may be undefined (i.e. $\infty$). This is precisely the case 
when $M = K$ is a separably closed, non algebraically closed field, and $X = k = \cap_{n} K^{p^{n}}$. $X$ is type-definable and minimal. For $p(x)\in S_{x}(M)$ the ``generic" type of $X$ as above, $RM_{X}(p) = 1$, but $RM(p) = \infty$, because otherwise there would be a {\em formula} in $p$ of ordinal valued Morley rank, and there are no such (nonalgebraic) formulas in the theory of separably closed fields.

\section{The $\sharp$ functor} 

\vlabel{sharp}
Here $K$ will be either a separably closed field of characteristic $p>0$
and finite degree of imperfection, or  a differentially closed field of characteristic $0$ (so with distinguished derivation $\partial$).  We distinguish the cases by ``characteristic $p$", ``characteristic $0$". In the characteristic $p$ case we will take $K$ to be say $\omega_{1}$-saturated, so as to be able to do model theory, although this will not always be necessary. Definability will mean in the sense of the structure $K$.
In the
characteristic $0$ case, as $DCF_{0}$ is $\omega$-stable we have $DCC$
on definable subgroups of a definable group, so any {\infdef} group is
definable. In the characteristic $p$ case, by stability,  any {\infdef} subgroup is an intersection of at most countably many definable groups.

\begin{definition} Let $G$ be a semiabelian variety over $K$. Then $G^{\sharp}$ is the smallest {\infdef} subgroup of $G(K)$ which is Zariski-dense in $G$.
\end{definition}

\medskip 
Various equivalent
characterizations of $\Gpinf$ were given in \cite{BD2}. 
In particular it was shown (\cite{BD2}, Proposition 3.6) that $\Gpinf$ is the unique divisible subgroup
of $G(K)$ which is Zariski-dense in $G$.
But the following
one was omitted at the time. 


\begin{proposition}\vlabel{Zardense} Suppose that $char(K)=p$ and let $G$  be a semiabelian variety over
  $K$. Then $\Gpinf$ is the smallest \infdef group of $G(K)$ which is
  Zariski dense in $G$, hence $\Gpinf=G^{\sharp}$. 
\end{proposition}
\begin{proof} Let $H$ be any \infdef subgroup of $G(K)$,  also Zariski
  dense in $G$. By stability, $H$ is a decreasing intersection of
  definable subgroups of $G(K)$,  $(H_i)_{i\in I}$. Certainly each
  $H_i$ is itself Zariski dense in $G$. By \cite{BD1} Corollary 4.16,
the connected component of $H_i$, $C_i$
  is also definable in $G(K)$ and has finite index in
  ${H_i}$. It follows that it is also Zariski dense in $G$. 

 Now,  
for every $r\geq 1$ the (definable) subgroup  $[p^r] C_i$ is also
Zariski dense in $G$. 
It follows by compactness and saturation,  that $\cap_{n\geq 1} [p^n] C_i$  is also 
Zariski dense in
$G$. But $\cap_{n\geq 1} [p^n] C_i$ is a divisible
group, and by the remark above,
$\Gpinf = \cap_{n\geq 1} [p^n] C_i$ for every $i$ and is hence
contained in  $H$. \end{proof}

In characteristic $0$, $G^{\sharp}$ is 
sometimes called the ``Manin
kernel" (see \cite{Marker-quaderni}). Alternative characterizations and key properties in arbitrary characteristic are given in the following Lemma.

\begin{lemma}\vlabel{othercharacterizations} (i) $G^{\sharp}$ can also be characterized as the smallest {\infdef} subgroup of $G(K)$ which contains the (prime-to-$p$, in char. $p$ case) torsion of $G$.
\newline
(ii) $G^{\sharp}$ is connected (no relatively definable subgroup of finite index), and of finite $U$-rank in char. $p$, and finite Morley rank in char. $0$.
\newline
(iii) If $G=(G_0)_K$ for some $G_0$ over the constants $\C$ of $K$, then $G^\sharp =
G(\cal C )$. 
\end{lemma}
\begin{proof} (i) Recall first that the (prime-to-$p$) torsion is contained in $G(K)$.  In the characteristic $p$ case, $G^{\sharp} = 
{\Gpinf}$ does contain the prime-to-$p$ torsion. On the other hand as
the prime-to-$p$ torsion is Zariski-dense in $G$ any subgroup of $G$
containing the prime-to-$p$ torsion is Zariski-dense. So the lemma is
established in characteristic  $p$.  The characteristic $0$ case is
well-known and due originally to Buium. See for example Lemma 4.2 of 
\cite{Pillay-countable} where it is proved that any definable Zariski-dense subgroup of a connected commutative algebraic group $G$ contains $Tor(G)$. 
\newline
(ii)  $G^\sharp$ is connected as any finite index subgroup of a
Zariski-dense subgroup is also Zariski-dense. In the characteristic 0
case, Buium \cite{Buium1} showed that $G^\sharp$ has finite Morley rank. 
 An account, using D-groups, appears in \cite{BePi}. 
In the
characteristic $p$ case, finite $U$-rank of $G^{\sharp}$ was first shown by Hrushovski in
\cite{Hrushovski}, and can also be seen to follow easily from  Lemma \ref{divweil2}.
\newline
(iii) In characteristic $p$, this is a direct consequence of Lemma
\ref{divweil2} or of the fact that $G(\C)$ is both divisible and Zariski-dense in $G$. In characteristic $0$
it can be seen as follows: Assume $G$ to be defined over $\C$. Note that $G(\cal C)$ is definable in the differentially closed field $K$. As $\C$ is algebraically closed
$G(\C)$ is Zariski-dense in $G$, hence $G^{\sharp}\subseteq G(\C)$. 
If $G^{\sharp} \subsetneq G(\C)$, $G^{\sharp}=H(\C)$ for some proper algebraic subgroup $H$ of $G$ over $\C$, and then $H(\C)$ could not be Zariski-dense in $G$.
\end{proof}

\begin{lemma}\vlabel{surjectivity}  Let $G$, $H$ be semiabelian varieties over $K$, and $f:G\to H$ a (not necessarily separable)
rational $K$-homomorphism. Then
\newline
(i)  $f(G^{\sharp}) \subseteq H^{\sharp}$.
\newline
(ii) If $f$ is dominant then $f(G^{\sharp}) = H^{\sharp}$.
\end{lemma}
\begin{proof} (i) Let $Tor_{p'}(G)$ be the prime-to-$p$ torsion  (so all the torsion in char. $0$). Note that
$f(Tor_{p'}(G))\subseteq Tor_{p'}(H)$. If (i) fails then $ D = f(G^{\sharp})\cap H^{\sharp}$ is a proper 
{\infdef} subgroup of $H(K)$ which by Lemma \ref{othercharacterizations} contains
$f(Tor_{p'}(G))$. But then $f^{-1}(D)\cap G(K)$ is an  
{\infdef} subgroup of $G(K)$ which contains $Tor_{p'}(G)$ and is properly contained in $G^{\sharp}$, contradicting
Lemma \ref{othercharacterizations}.
\newline 
(ii) Note that $f(G^{\sharp})$ is {\infdef} (by $\omega_1$-saturation in characteristic $p$, since in this case $f(\cap G_i)=\cap f(G_i)$), and since $f$ 
is dominant,
$f(G^{\sharp})$ 
must be Zariski-dense in $H$. By part (i), and the definition of $H^{\sharp}$, $f(G^{\sharp}) = H^{\sharp}$.
\end{proof}

\begin{remark} (Characteristic $p$) Let $f:G\to H$ be as in the hypothesis of Lemma \ref{surjectivity}
  (ii). If $f$ is separable (that is induces a separable extension of
  function fields) then as we remarked in Fact \ref{separablemorphisms} 
$f_{|G(K)}: G(K) \to H(K)$ is surjective. 
If $f$ is not separable, $f$ may no longer be surjective at the level of
  $K$-rational points, but nevertheless Lemma \ref{surjectivity}(ii)
  says it is surjective on the $\sharp$-points when $K$ is
  $\omega_1$-saturated.\\
Note however that if $f$ is an isogeny, $f(p^{\infty}G(K))=p^{\infty}H(K)$ without any saturation assumption (if $f$ has degree of inseparability $n$, then $[p^n]H(K)\subseteq f(G(K))$, and one concludes by K\"onig's lemma).
\end{remark}

\vspace{5mm}
\noindent
By Lemma \ref{surjectivity} (i) 
we can consider $\sharp$ as a functor from the category of semiabelian varieties over $K$ to the category of {\infdef} groups in $K$. It is natural to ask whether $\sharp$ preserves exact sequences, and this is an important theme of the paper.

Recall that by an {\em exact sequence of algebraic groups} defined over $K$,
we  mean that the homomorphisms are not only over $K$ but also
separable. Consider two semiabelian
varieties $G_{2}, G_{3}$ over $K$, a separable surjective
$K$-homomorphism $f:G_{2}\to G_{3}$, with $Ker(f)
= G_{1}$ connected and thus a semiabelian subvariety of $G_{2}$
over $K$.  Then, by Fact \ref{separablemorphisms} the sequence $0 \to G_1(K) \to G_2(K) \to G_3(K) \to 0$
clearly remains  exact (in the category of definable groups in $K$). By Lemma \ref{surjectivity} the sequence  
$$0 \to G_{1}^{\sharp} \to G_{2}^{\sharp} \to G_{3}^{\sharp} \to 0$$
will be exact if and only if  $$G_{1}^{\sharp} = G_{1}(K)\cap G_{2}^{\sharp}.$$

\noindent So the group  $(G_{1}(K)\cap
G_{2}^{\sharp})/G_{1}^{\sharp}$ is the obstruction to exactness. 

In the
characteristic $0$ case this group which is clearly of finite Morley
rank, can be seen to be connected and embeddable in a vector group. By
Lemma 4.2 
of \cite{Pillay-countable} for example, $G_1(K)/G_1^{\sharp}$ (as a group
definable in $K$ by elimination of imaginaries) embeds definably 
in $(K,+)^{n}$ for some $n$. Hence $(G_1(K) \cap G_2^{\sharp})/G_1^{\sharp}$ also embeds in $(K,+)^{n}$, and as such is a (finite-dimensional) vector space over the field of constants of $K$. Hence
$(G_1(K) \cap G_2^{\sharp})/G_1^{\sharp}$ is connected. Note that, as
$G_1^{\sharp}$ 
is also connected, it follows that
$G_1(K) \cap G_2^{\sharp}$ itself is also  connected.

The
characteristic $p$ case is different in an interesting way. Note first,
that  the group $(G_{1}(K)\cap
G_{2}^{\sharp})/G_{1}^{\sharp}$ is not even infinitely definable, it is the quotient
of two  \infdef  groups. Such 
groups are usually called ``hyperdefinable''.

We will recall the (model theoretic) definition of a connected
component. First, if 
$G$ is an \infdef group in a stable theory, then we have DCC on intersections of uniformly relatively definable subgroups (see \cite{poizat} or \cite{wagner}). What this means is that
if $\phi(x,y)$ is a formula, then the intersection of any collection of subgroups of $G$ relatively defined by some instance of $\phi(x,y)$, is a finite subintersection. It follows that, working in a saturated model say, the intersection of all relatively definable subgroups of $G$ of finite index, is the intersection of at most $|L|$ many (where $L$ is the language). We call this intersection, $G^{0}$, the {\em connected component} of $G$. It is normal, and type-definable over the same set of parameters that $G$ is. Moreover $G/G^{0}$ is naturally a profinite group. 
In the $\omega$-stable case (or the relative Finite
Morley Rank case as in section \ref{Morleyrank}), by DCC on relatively definable
subgroups, $G^0$ will itself be relatively definable and of  finite index in $G$ . 


\begin{lemma} \vlabel{connectedcomponent}(Characteristic $p$) Let $G_{1}$ be a semiabelian subvariety of the semiabelian variety $G_{2}$,
both defined over $K$. Then $G_{1}^{\sharp}$ is the connected component of $G_{1}(K)\cap G_{2}^{\sharp}$.
\end{lemma}
\begin{proof} First by \ref{surjectivity},  $G_{1}^{\sharp}$ is a
  subgroup of $G_{1}(K)\cap G_{2}^{\sharp}$. By Lemma
  \ref{othercharacterizations}   $G_{1}(K)\cap G_{2}^{\sharp}$ is
      {\infdef} of finite $U$-rank.  Hence, for any $H$ \infdef subgroup
  of $G_{1}(K)\cap G_{2}^{\sharp}$, classical $U$-rank
  inequalities for groups give us that $U(H[n])  + U([n] {H}) = U(H)$.
  As for each $n$ the
      $n$-torsion of $H$  is finite, $U([n]H)= U(H)$.  It follows
  that $H$ is
  connected iff it is  divisible: If $H$ is connected, then any proper
  infinitely definable subgroup of $H$ has strictly smaller
  $U$-rank than $H$, so
  for every $n$,  $[n]H = H$, and $H$ is divisible. But $G_{1}^{\sharp}$
  is  
the biggest divisible subgroup of $G_{1}(K)$. Thus $G_{1}^{\sharp}$ must coincide with the connected component of $G_{1}(K)\cap G_{2}^{\sharp}$.
\end{proof}

\begin{remark} By Lemma \ref{connectedcomponent},  the quotient  $(G_{1}(K)\cap
  G_{2}^{\sharp})/G_{1}^{\sharp}$ is a profinite group. If
  $G_{2}^{\sharp}$ had relative Morley rank, the quotient would have to
  be finite (as remarked above). We will see in section \ref{charp}
  an example where the quotient is infinite and   give an 
explicit description of
  this  quotient in terms of suitable Tate modules.
\end{remark}

For the record we now mention cases (in characteristic $p$) where $G^{\sharp}$ has (finite) relative Morley rank.
\begin{fact} \vlabel{caseswithRMR}(Characteristic  $p$). Let $G$ be a semiabelian variety over $K$. Then
\newline
(i) If $G$ descends to ${\Kpinf}$ (in particular if $G$ is an algebraic torus) then $G^{\sharp}$ has finite relative Morley rank.
\newline
(ii) If $G = A$ is an abelian variety then $A^{\sharp}$ has finite relative Morley rank.
\end{fact}
\begin{proof} 
(i) We may assume that $G=(G_0)_K$ for some $G_0$ over ${\Kpinf}$. Then by
  Lemma \ref{divweil2} $G^{\sharp} = {\Gpinf} = G({\Kpinf})$. As $\Kpinf$ is a
  ``pure" algebraically closed field inside $K$, $G({\Kpinf})$ has relative
  Morley rank equal to the (algebraic) dimension of $G$. 
\newline
(ii) The abelian variety $A$ is isogenous to a product of simple abelian
varieties. So we may reduce to the case where $A$ is simple. In that
case $A^\sharp$ has no proper infinite definable subgroup (2.16 in
\cite{Hrushovski} or  Cor.3.8 in
\cite{BD2}). By stability, $A^\sharp$ has no proper infinite {\infdef} subgroup. 
We will now use an appropriate version of Zilber's indecomposability theorem to see that 
$A^{\sharp}$ has finite relative Morley rank. As $A^{\sharp}$ has finite $U$-rank, there is some small submodel $K_{0}$ (over which $A^{\sharp}$ is defined) and a complete type $p(x)$ over $K_{0}$ extending ``$x\in A^\sharp$", which has $U$-rank $1$ (and is of course stationary). Let $Y\subseteq A^\sharp$ be the set of realizations of $p$. Then $Y$ is an 
{\infdef} subset of $A^\sharp$ which is ``minimal", namely $Y$ is infinite and every 
relatively definable subset of $Y$ is either finite or cofinite.  We claim that $Y$ is ``indecomposable" in $A^\sharp$, namely for each relatively definable subgroup $H$ of $A^\sharp$, $|Y/H|$ is $1$ or infinite. For if not, then as remarked earlier, the intersection of all the images of $H$ under automorphisms fixing $K_{0}$ pointwise will be a finite subintersection $H_{0}$, now defined over $K_{0}$, and we will have $|Y/H_{0}| >1$ and finite, contradicting stationarity (or even completeness) of $p$. Let now $X$ be a translate of $Y$ which contains the identity $0$. Then $X$ is still a minimal {\infdef} subset of $A^\sharp$. Moreover
Theorem 3.6.11 of \cite{wagner} or Theorem 6.10 of \cite{poizat}  apply  to this situation, to yield that the subgroup $B$ say of $A^\sharp$ which is generated by $X$ is  {\infdef} and moreover 
of the form $X + X + ... + X$  ($m$ times) for some $m$. As noted above, it follows that $B = A^\sharp$, and so the function $f:X^{m} \to A^\sharp$ is a definable surjective function between {\infdef} sets, in the sense of section \ref{Morleyrank}. But as $X$ is minimal, clearly $RM_{X}(X) = 1$ and $RM_{X^{m}}(X^{m}) = m$. By Lemma \ref{mapRMR} (ii), $A^\sharp$ has finite relative Morley rank too. 
\end{proof} 

\vspace{5mm}
Let us remark that, in the context of the proof of (ii) above, when $A$ is a simple abelian variety over $K$ which does not descend to the constants, then via the dichotomy theorem for minimal types in separably closed fields, $A^{\sharp}$ is connected, of $U$-rank 1 hence has relative Morley rank $1$. However we wanted to avoid the appeal to the dichotomy theorem, and hence above we use the proof involving a version of Zilber's indecomposability theorem.

\section{Characteristic $p$}

\vlabel{charp}

Here we follow the ``naive'' algebraic approach which works in
  a very simple way in 
  characteristic $p$. In order to highlight the uniformity with char. 0
  we will, in the next  section, adopt a point of view closer to algebraic geometry.

We deal now with the characteristic $p$ case. Let $G$ be a semiabelian
variety over any  model $(K,\partial)$ of $CHF_{p,1}$, that is any
separably closed field of degree of imperfection $1$.

\subsection{Torsion points, Tate modules and descent}\vlabel{TorsionTate}


We make no saturation hypothesis for the moment.

\begin{definition} \vlabel{Gtilde}
We define $\tilde G$ as the inverse limit
$$
\tilde{G}:=\mathop{\text{lim}}_{\leftarrow} (G\stackrel{[p]}{\longleftarrow} G \stackrel{[p]}{\longleftarrow} \ldots ).
$$
In particular, for $L$ an extension of $K$ (we will mainly consider $L=K$ or $L=\overline K$),
$$
\tilde{G}(L)=\{(x_i)_{i\in \N}\in G(L)^{\N}\, :\, \forall i\ge 0, x_i=[p]x_{i+1}\}.
$$
Let $\pi_G$ be the projection onto the ``left component'' $G(L)$. The
kernel of $\pi_G$ is precisely $T_pG(L)$, where $T_pG$ is the Tate module of $G$.\\
Its $L$-points in an arbitrary algebraically closed extension $L$  of $K$ coincide with the sequences of torsion points in $\overline K$, 
$$ T_pG(\overline K)=\{(x_i)_{i\in \N}\in G(\overline K)^{\N}\, : \, x_0=0,\forall i\ge 0, x_i=[p]x_{i+1}\}
$$
\end{definition}

By definition, $\tilde G$ is a proalgebraic group, i.e. an inverse limit of algebraic groups. In section \ref{Gtildecharp}, $\tilde G$
will be viewed as a group scheme. Here we adopt a more naive point of view, closer to model theory. 
Objects such as $\tilde G(K)$ and $T_pG(K)$ are what are called
``$*$-definable" groups in $K$ (projective limits of definable groups).

Let us note that for a given $g_0\in G(K)$, $g_0 \in G^{\sharp}$ if and only if there is some $(x_i)_{i\in
  \N}\in \tilde{G}(K)$ with $g_0 = x_0$; we deduce directly from this  the relation between the Tate module of $G$ and $G^{\sharp}$.
\begin{lemma} \vlabel{sharpTate}
The morphism $\pi_G$ induces an exact sequence of $*$-definable groups.
$$
0 \rightarrow T_pG(K) \rightarrow \tilde G(K) \stackrel{\pi_G}{\rightarrow} G^{\sharp} \rightarrow 0.
$$
\end{lemma}


In the case of ordinary semiabelian varieties, if the dimension of the
abelian part is $a$,  $T_pG(\overline K)\simeq
\Z_p^a$ (see Fact \ref{torsion}).
We relate now the part of the $p^{\infty}$-torsion lying in $K$ with
issues of descent. Most  of the following results are certainly well-known,
see for example \cite{voloch} 
for the description of the torsion of $G$
for abelian schemes of maximal Kodaira-Spencer rank. But we have  found
no systematic exposition which we could quote and furthermore, we choose to
give here very elementary proofs which are suitable for our
purpose.

\begin{proposition}\vlabel{abelianseparabletorsion} Let $G$ be an
  ordinary semiabelian variety over  $K$. Then for every $n$, $G[p^n](\overline K) =
  G[p^n](K) $ if and only if $G$ descends to $K^{p^n}$. In particular,
  $G$ descends to $K^{p^\infty}$ if and only if $G[p^\infty](\overline K) = G[p^\infty](K)$
if and only if $T_pG(\overline K)=T_pG(K)$.
\end{proposition}
\begin{proof} Let us fix $n\ge 1$. If $G$ descends to $K^{p^n}$, we may
  assume that $G=(G_0)_K$ for some $G_0$ over $K^{p^n}$. Since $G_0$ is ordinary, $V_n$ is separable and the geometric points of the kernel of $V_n$ are
$K^{p^n}$-rational, and since $[p^n]=V_n \circ F^n$, $G_0[p^n](\overline K)=F^{-n}(\text{Ker}(V_n)(\overline K))\subseteq G_0(K)$.\\
Conversely, assume that $G[p^n](\overline K)\subseteq G(K)$. Since $V_n$ is separable, $G$ is isomorphic to the quotient $F^n G/\text{Ker}(V_n)$. But $\text{Ker}(V_n)(\overline K)=F^n(G[p^n](\overline K))$ is a finite group of $K^{p^n}$-rational points, hence
$F^nG/\text{Ker}(V_n)$ descends to $K^{p^n}$.\\
The ``in particular'' statement follows from Proposition \ref{modulispace}.
\end{proof}

\smallskip

\begin{corollary}\vlabel{allptorsion} Let $K_0 $ be an algebraically closed  field and
  $K_1 >K_0$ a finitely generated extension of $K_0$. Let $G$ be an ordinary
  semiabelian variety over $K_1$. If $G[p^\infty](\overline{K_1}) = G[p^\infty](K_1)$,
  then $G$ descends to $K_0$. \end{corollary}
\begin{proof} As $K_0$ is algebraically closed, $K_1$ is a separable
  extension of $K_0$, hence it is contained in the separable closure
  of $K_0(t_1,\ldots, t_n)$ for $t_1,\ldots,t_n$ algebraically
  independent. 
Then (Fact \ref{invariant1}) there is a separably
  closed field $K$ of degree of imperfection $1$, extending $K_1$ and
  such that $K_0 = K^{p^\infty}$. 
We can now apply Proposition \ref{abelianseparabletorsion} to conclude
  that $G$ descends to $K^{p^\infty}$. \end{proof}

This yields easily the following result 
(compare with Fact \ref{easyseparableisogeny}, here $f$ is no longer separable but $G$ is ordinary).

\begin{corollary}\vlabel{ordinaryisogenydescent} Let $G$ be an ordinary semiabelian variety over some
  algebraically closed field $K_0$. If $H$ is any semiabelian variety over
  $K_1 > K_0$ such that there is an isogeny $f$ from $G_{K_1}$ onto $H$, then
  $H$ descends to $K_0$. \end{corollary}
\begin{proof} Let $K_2 <K_1$  be a finitely generated extension of $K_0$ over
 which $H$ and the isogeny $f$ from $G$ to $H$ are defined. We claim
 first 
 that any point of $p^\infty $-torsion in $H$ is the image of a point
 of $p^\infty$-torsion in $G$: indeed let $h \in H[p^\infty](\overline{K_2})$,
 i.e. for some $m$, $[p^m] h = 0$. Let $g \in G(\overline{K_2})$, be a
 preimage of $h$, $f(g) = h$. Then $[p^m] g  \in Ker f$. 
Let $n\ge 1$ be the order of the  finite
 group $(Ker f)(\overline{K_2})$. Then $n = p^r d$, where $d$ is
 prime to $p$. Write $1 = u d + v p^m$, $u,v \in \mathbb Z$. Then
 $g = [ud] g + [vp^m] g$, so $h=f(g) = f([ud]g)$, with $[p^{r+m}] [ud] g
 = 0$.\\
Now as $K_0$ is   algebraically closed, 
 $G[p^\infty](\overline{K_2}) = G[p^\infty](K_0 )$ and hence by the above claim $ H[p^\infty](\overline{K_2})
\subseteq f(G[p^{\infty}](\overline{K_2})) = f(G[p^{\infty}](K_0))\subseteq H[p^{\infty}](K_2)$.
We can now apply Corollary \ref{allptorsion}. 
\end{proof}

\begin{corollary}\vlabel{abeliandescent} Let $K_0 < K_1$ be an extension of fields with $K_0$ algebraically closed and
 let $0 \longrightarrow C \longrightarrow
  B\longrightarrow A\longrightarrow 0$, be an exact sequence of
 ordinary  abelian varieties over $K_1$, such that $A$ and $C$
 descend to $K_0$. Then $B$
 descends to $K_0$.    \end{corollary}
\begin{proof} By Poincar\'e reducibility theorem, $B$ is isogenous to $A\times C$, which descends to $K_0$, and we just have to
apply Corollary \ref{ordinaryisogenydescent}. \end{proof}

\begin{remark}See Remark \ref{supersingular} for a counterexample if one does not require the abelian varieties to be ordinary.
\end{remark}

We complete this section with some basic remarks about torsion in
$G(K)/G^{\sharp}$ $(=G(K)/p^{\infty}G(K))$ in characteristic $p$ which will immediately  enable us to
describe the link between the question of relative Morley rank  and that of preservation of exactness. 

\begin{lemma} \vlabel{torsionfree}  Let $G$ be a semiabelian variety over $K$.
\newline
(i) $G[p^{\infty}](K)$ is a direct sum of a divisible group and a finite group.
\newline 
(ii) $G(K)/p^{\infty}G(K)$ has finite torsion.
\newline 
(iii) If $G$ descends to ${\Kpinf}$ then $G(K)/p^{\infty}G(K)$ is torsion-free.
\newline
(iv) If $G(K)$ has trivial $p$-torsion then $G(K)/p^{\infty}G(K)$ is torsion-free.
\end{lemma}
\begin{proof}
(i) $G[p^{\infty}](K)$ is a subgroup of $G[p^{\infty}](\overline K)$ which is a finite direct sum of copies of the Pr\"ufer group 
${\mathbb Z}_{p^{\infty}}$.
\newline 
As  $p^{\infty}G(K)$ is divisible, if $g\in G(K)$ and $ng\in p^{\infty}G(K)$
then there is $h\in p^{\infty}G(K)$ so 
that $ng = nh$ whereby $n(g-h) = 0$ so 
$g$ is congruent mod $p^{\infty}G(K)$ to an element of order $n$. We know that $p^{\infty}G(K)$ contains all the prime-to-$p$-torsion of $G$. On the other hand by (i) 
$G[p^{\infty}](K)/p^{\infty}G(K)$ is finite. This gives (ii) immediately.  
\newline
Similarly, for  cases (iii) and (iv), where $p^{\infty}G(K)$ contains all the torsion of $G(K)$. \end{proof}

\begin{proposition} \vlabel{Morleysemiabelian} 
  Suppose that $K$ is $\omega_1$-saturated and let  $G$ be a
  semiabelian variety over $K$, 
$0 \rightarrow T \rightarrow G
\rightarrow A \rightarrow 0$. Then the following are equivalent: 
\newline (i) $G^\sharp$  has relative
Morley rank \\
(ii) the sequence $0 \rightarrow T^{\sharp} \rightarrow G^\sharp \rightarrow A^\sharp \rightarrow 0$
is exact \\
(iii) $ (T(K) \cap G^{\sharp})/ T^{\sharp}$ is finite\\
(iv) $T(K) \cap G^{\sharp} $ is divisible. 
\end{proposition}
\begin{proof} 
By the previous lemma, as $T$ has no $p$-torsion,
$T(K)/T^{\sharp}$ is torsion free.  Also note that $T^{\sharp} = T({\cal C})$ is divisible and is the connected component of $T(K) \cap G^{\sharp}$ (\ref{connectedcomponent}).  Hence $(T(K) \cap G^{\sharp})
/T^{\sharp}$ is finite iff it is trivial iff the sequence  
 $0 \rightarrow T^{\sharp} \rightarrow G^\sharp \rightarrow A^\sharp \rightarrow 0$
is exact. And moreover these conditions are equivalent to the divisibility of $T(K) \cap G^{\sharp}$.
This gives the equivalence of (ii), (iii), and (iv). 
\newline
On the other hand if $G^{\sharp}$ has finite relative Morley rank, then every relatively definable subgroup is connected by finite, so (i) implies (iii).  
Conversely, we have seen (\ref{caseswithRMR}) that both $T^{\sharp}$ and $A^{\sharp}$ have
relative Morley rank. By \ref{RMRexact}, the exactness of the sequence
implies that  $G^{\sharp}$ also has relative Morley rank. Thus (ii) implies (i).
\end{proof}

\subsection{Exactness and descent}
\vlabel{exactnessdescent}

\noindent 

We now assume that $K$ is $\omega_1$-saturated.

\begin{proposition}\label{tate}
Let $0\longrightarrow G_1 \stackrel{h}\longrightarrow G_2\stackrel{f}
\longrightarrow G_3
\longrightarrow 0$  be an exact sequence of semiabelian varieties over
$K$. Let $f^{\omega}$ be the induced
morphism of ${\mathbb Z}_p$-modules from $T_pG_2(K)$ to
$T_pG_3(K)$. Then there is an isomorphism $$
\phi : (h(G_1(K)) \cap G_2^\sharp) /h(G_1^\sharp)
\widetilde{\longrightarrow} T_pG_3(K)/f^{\omega}(T_pG_2(K)). $$
\end{proposition}
\begin{proof} Note that we may assume that $h$ is the
  inclusion. We first define $\phi$:\\
Let $g$ be in $G_1(K)\cap p^{\infty}G_2(K)$. There exists an element $(g_i)_{i\in \N}$ in $T_p(G_2(K),g)$ (the fiber of $\tilde G_2(K)$ over $g$), with $g_0=g\in G_1(K)$. Hence $f^{\omega}((g_i))\in T_pG_3(K)$. We check that it gives a well-defined map (even a group homomorphism) from $G_1(K)\cap p^{\infty}G_2(K)$ to $T_pG_3(K)/f^{\omega}(T_pG_2(K))$: if $(g'_i)_{i \in N}$
is another element in $T_p(G_2(K),g)$, then $(g_i)-(g'_i)\in T_pG_2(K)$ hence $f^{\omega}((g_i))-f^{\omega}((g'_i))\in f^{\omega}(T_pG_2(K))$. Let us prove now that the kernel of this map is $p^{\infty}G_1(K)$: if $g\in p^{\infty}G_1(K)$, we can choose $(g_i)\in T_p(G_1(K),g)$, which is sent to $0$ by $f^{\omega}$. Conversely, assume
that for some $(g_i)\in T_p(G_2(K),g)$ and some $(h_i)\in T_pG_2(K)$, $f^{\omega}((g_i))=f^{\omega}((h_i))$. Then $(g_i-h_i)\in T_p(G_1(K),g)$, which gives that
$g\in p^{\infty}G_1(K)$.\\
Hence we have obtained an embedding $\phi: (G_1(K) \cap p^{\infty}G_2(K))/p^{\infty}G_1(K) \hookrightarrow T_pG_3(K)/f^{\omega}(T_pG_2(K))$. It remains to prove that it is surjective. For $(h_i)_{i\in \N}\in T_pG_3(K)$, 
we can realize in $K$ (which is $\omega_1$-saturated) the following type of length $\omega$ over $K_0((h_i))$ ($K_0$ is a countable subfield of definition):
$$
\bigwedge_{i\in \N} (x_i\in G_2 \wedge f(x_i)=h_i \wedge x_i=[p]x_{i+1}).
$$
(It can be realized for $i\le n$ by choosing some $g_{n+1}\in G_2(K)$ such that $f(g_{n+1})=h_{n+1}$, and then defining $g_i=[p^{n+1-i}]g_{n+1}$.)
For a realisation $(g_i)_{i\in \N}$ of this type, we have $g_0\in G_1(K)$ (since $f(g_0)=h_0=0$), $(g_i)\in T_p(G_2(K),g_0)$, hence $g_0\in p^{\infty}G_2(K)$ and
$f^{\omega}((g_i))=(h_i)$.\end{proof}

\begin{remark} \vlabel{Tpsurj}
It follows in particular that the sequence $0 \to G_1^{\sharp} \to G_2^{\sharp} \to G_3^{\sharp} \to 0$ is exact if and only if
the map $f^{\omega} : T_pG_2(K) \to T_p G_3(K)$ is surjective.
\end{remark}

\begin{proposition}\vlabel{Maindescentcarp}
Let 
$0\rightarrow G_1 \rightarrow G_2 \rightarrow G_3\rightarrow 0$ be an
  exact sequence of ordinary  semiabelian varieties over $ K $. Suppose
  that $G_1$ and $G_3$ descend to the constants of $K$.\\
Then, the sequence $0 \to G_1^{\sharp} \to G_2^{\sharp} \to G_3^{\sharp} \to 0$
remains  exact if and only if  $G_2$ also descends 
  to the constants.  
\end{proposition}
\begin{proof} Here again we may assume that the map $G_1\to G_2$ is the inclusion.
First note that the $\sharp$ sequence is exact if and only if $G_1(K) \cap G_2^{\sharp}=G_1^{\sharp}$ by Lemma \ref{surjectivity}.\\
Let $K_0$ be a countable elementary submodel of $K$ over which
everything is defined. By isomorphism, we can suppose that both $G_1$
and $G_3$ are actually defined over ${\cal C}\cap K_0$, the field of constants
of $K_0$ (precisely $G_i=(G'_i)_K$ for some $G'_i$ over ${\cal C} \cap K_0$, $i=1,3$).\\ 
If $G_2$ descends to the constants, then by isomorphism, we can suppose
that $G_2=(G'_2)_K$ for some $G'_2$ over the constants, so for every $i$
${G_i}^\sharp = G'_i ( {\cal C})$. And  then $G_1(K) \cap G_2^\sharp =
G'_1(K) \cap G'_2 ({\cal C}) = G'_1 ({\cal C}) = G_1^\sharp$. 

\noindent For the converse, suppose that $0\rightarrow
{G_1}^\sharp  \rightarrow {G_2}^\sharp  \rightarrow {G_3}^\sharp
\rightarrow 0$ is exact.\\  
Our assumption that the $G_i$'s are ordinary
ensures that for each $i$, $T_p G_i (\overline K) \cong {{\mathbb Z}_p}^{a_i}$,
where $a_i$ is the dimension of the abelian part of $G_i$. If $G_1$
and $G_3$ descend to ${\cal C}$, then $T_p G_1 (K) = T_p G_1 ({\cal C})=T_p G_1
(\overline K ) $ and
$T_p G_3 (K) =T_p G_3 ({\cal C})=  T_p G_3 (\overline K )$. 
By Remark \ref{Tpsurj}, the sequence  $$0\longrightarrow T_p {G_1}(K)
\longrightarrow T_p G_2 (K)  \longrightarrow
T_p G_3 (K) \longrightarrow 0$$ is exact. It follows that $T_p G_2 (K)
\cong {{\mathbb Z}_p}^{a_1 + a_3}$. As $a_1 + a_3 = a_2$ (by exactness
of $0\longrightarrow G_1 \longrightarrow G_2 \longrightarrow
G_3\longrightarrow 0$), and as $T_p G_2 (K)$ is a direct factor
submodule of $T_p G_2(\overline K)$, it follows that $T_p G_2 (K) = T_p G_2
(\overline K)$, and by Proposition \ref{abelianseparabletorsion}, that
$G_2$ descends to  the constants.\end{proof}

\begin{corollary}\vlabel{nonexactexamplecarp} 
For any ordinary 
abelian variety $A$ defined over the
  constants of $K$,  there exists an exact sequence over $K$, 
$$  0 \longrightarrow {\mathbb G}_m \longrightarrow H \longrightarrow A_K \longrightarrow 0 $$  such that 
$$  0 \longrightarrow {\mathbb G}_m^{\sharp} \longrightarrow H^{\sharp} \longrightarrow (A_K)^{\sharp} \longrightarrow 0$$
is not exact. 
\end{corollary}
\begin{proof} We use the fact that  $EXT(A, {\mathbb G}_m)$ is parametrized
  (up to isomorphism) by the dual abelian variety of $A$, say $\hat A$,
  which is also over the constants, as in Proposition \ref{modulispace}. Then $H$ will descend to the
  constants $\cal C$ of $K$ if and only if $H$ corresponds to a $\cal C$-rational
  point of $\hat A$. So just pick some $K$-rational point of $\hat A$
  which is not $\cal C$-rational.\end{proof}

We have established in Proposition \ref{Morleysemiabelian} the connection between
exactness and relative Morley rank, and we can conclude that: 

\begin{corollary} \vlabel{counterexamplecarp} There
  is an ordinary 
  semiabelian variety $G$, such that $G^{\sharp}$ does not have relative
  Morley rank.\end{corollary}
 
In fact,  as above, for any ordinary  abelian variety $A$ defined over
$K^{p^\infty}$, there is some semiabelian variety $G$ in
$EXT(A,\mathbb{G}_m)$ such that $G^\sharp$ does not have relative Morley
rank.

\medskip

We will finish this section with some direct corollaries of Proposition \ref{tate}. Again,
$0\longrightarrow G_1 \longrightarrow G_2 \stackrel{f}\longrightarrow 
G_3\longrightarrow 0$ is an  exact sequence of  semiabelian varieties
over $ K$, with $G_1 \to G_2$ the inclusion map. Recall from Proposition \ref{tate} that
$(G_1(K) \cap G_2^{\sharp}) /G_1^\sharp\, \cong \, T_pG_3 (K) / f^{\omega} (T_p G_2 (K))$.   

\begin{corollary}\vlabel{finitetorsionin$G_3$carp} If
  $G_3[p^\infty](K)$ is finite, then the $\sharp$ sequence  is exact.
\end{corollary}
\begin{proof} Since $G_3[p^{\infty}](K)$ is finite, $T_pG_3(K)=0$.\end{proof}

\medskip




If we add the assumption that the  semiabelian varieties have relative
Morley  rank, we  get the following characterization: 


\begin{proposition}\vlabel{torsioninkernel} 
 Let $0 \longrightarrow G_1 \longrightarrow G_2 \longrightarrow G_3 \longrightarrow 0$ be
  an exact sequence of semiabelian varieties over $K$ such that $G_2^\sharp$ 
 has relative Morley rank. Then the following are equivalent

(1) the sequence $0 \longrightarrow {G_1}^{\sharp} \longrightarrow {G_2}^\sharp \longrightarrow {G_3}^\sharp \longrightarrow 0$
is exact 

(2) $ G_1 [p^\infty ] (K) \cap {G_2}^\sharp \, = \,  G_1 [p^\infty
    ] (K) \cap   {G_1}^\sharp$. 


\noindent In particular the $\sharp$ sequence will be exact when $G_1$
descends to the constants, or, more generally, when 
$ G_1 [p^\infty ](\overline K ) = G_1[p^\infty ]    (K)$,  and also when
 $G_1[p^\infty ] (K) = {0}$. \end{proposition}
\begin{proof} Recall that ${G_i}^\sharp = p^\infty G_i (K)$. 
  We know that (1)  holds if and only if ${G_1}^\sharp = G_1(K) \cap {G_2}^\sharp$. So trivially, (1)  implies (2). We know that ${G_1}^\sharp$
contains all the $p'$-torsion of $G_1(K)$. It follows that if (2) holds, then 
$(G_1(K) \cap {G_2}^\sharp)/ {G_1}^\sharp $ is torsion free. 
As by assumption ${G_2}^\sharp $ has relative Morley rank, this  quotient
must be finite, if it is torsion free, it is trivial. 

If $ G_1 [p^\infty ](\overline K ) = G_1[p^\infty ]    (K)$ then
$G_1[p^\infty ]    (K) \subset G_1^\sharp$ and the conclusion holds. 
\end{proof}

\subsection{Further examples} \vlabel{further}

We will see in section  \ref{unifstat} that  in characteristic $0$, the $\sharp$ functor
  preserves exact  sequences of
  abelian varieties. This is not the case in characteristic $p$, even for
  ordinary abelian varieties.

 
The examples  of non exactness for abelian varieties will have 
to be quite different from the examples
seen in the previous section for semiabelian varieties, as can be seen from the following direct
corollary of Proposition \ref{torsioninkernel}. Recall from Fact
\ref{caseswithRMR} that for all abelian varieties $A$, $A^\sharp$ has
finite relative Morley rank. 

\begin{corollary}\vlabel{noptorsionabelian} We assume that $K$ is $\omega_1$-saturated. 
Let $0 \longrightarrow C \longrightarrow
  B\longrightarrow  A\longrightarrow 0$, be an exact sequence of
  abelian varieties over $K$. If  $C(K)$ has no $p$-torsion, or if $C$
  descends to the constants, then 
the sequence $0 \longrightarrow C^\sharp \longrightarrow
  B^\sharp \longrightarrow  A^\sharp \longrightarrow 0$ is exact.\end{corollary}

\begin{remark}\vlabel{counterexampletoMaindescent} From Corollary
  \ref{noptorsionabelian} and the example given in Remark \ref{supersingular}, we see that 
  Proposition \ref{Maindescentcarp} does not hold for non ordinary
  (semi)abelian varieties. 
\end{remark}


\medskip 

There are still cases, not covered by Corollary \ref{noptorsionabelian},
where one obtains non exactness, even in the ordinary case:

\begin{proposition}\vlabel{nonexactabelian}There is an exact sequence
  of (ordinary) abelian varieties such that the induced  $\sharp $ sequence is not exact.\end{proposition} 
\begin{proof}  Let $A$ be an ordinary elliptic curve, defined over
  $K^p$, which does not descend to $K^{p^\infty}$  and $C$ an ordinary
  elliptic curve defined over $K^{p^\infty}$.
Then we know by Proposition \ref{abelianseparabletorsion}  that 
$A [p](K)  \cong \Z/p\Z \cong C[p](K)$ but $A[p^\infty ](K)$ is finite. Pick an isomorphism $f$ between $A [p](K)$
  and $C[p](K)$.\\
Let $H \subset A[p](K) \times C[p](K)  := \{ (a, - f(a)) \, : \, a \in A|p][K)\}$, and $B := (A_K
  \times C_K) / H$. 
Then $A_K$ is isomorphic to $A_1 := (A_K \times {0} + H) / H \subset B$. 
Consider the exact sequence: 
$$ 0 \longrightarrow A_1 \longrightarrow B \stackrel{g}\longrightarrow  B/A_1
\longrightarrow 0.$$
Note that $C_1 := B/A_1$ is isogenous to $C_K$, hence by
\ref{ellipticdescent} or \ref{ordinaryisogenydescent}, descends to
$K^{p^\infty}$.\\
One can check that the $p^\infty$ sequence is no longer exact, more precisely, that
$p^\infty A_1 (K) \not= A_1(K) \cap p^\infty B(K) $. \\
Furthermore, if $K$ is $\omega_1$-saturated, by applying Proposition \ref{tate}, one sees
that $(A_1(K) \cap p^{\infty} B(K))/p^{\infty}A_1(K)$ is isomorphic to ${\mathbb Z} / p {\mathbb Z}$. \end{proof}










\medskip

\begin{remark}\vlabel{supersingular} 
The following example illustrates the necessity of the separable hypothesis 
in Facts \ref{easyseparableisogeny} and \ref{hardseparableisogeny}, and of the ordinary hypothesis
in Corollary \ref{abeliandescent} and Proposition \ref{Maindescentcarp}.\\
Let $E$ be a supersingular elliptic curve over $K$ ($\omega_1$-saturated), necessarily descending to $\overline \Fp$.
 For any
  abelian  variety $A$ there is a one-one correspondence  between (isomorphism classes
  of) 
purely inseparable  isogenies and sub $p$-Lie algebras of Lie $A$ (see \cite{Serre} or \cite{Mumford}).
 It follows that there is an abelian
 variety $A$ over $K$, isogenous to  $E\times E$, which cannot be isomorphic to 
any abelian variety defined over  $\overline {\mathbb
   F_p}$.\\
Furthermore, for such an $A$, it is easily seen that $A$ lies in $EXT(E_1,E_2)$ for some elliptic curves
$E_1$ and $E_2$ descending to $\overline \Fp$, and in this case $0 \to E_1^{\sharp} \to A^{\sharp} \to E_2^{\sharp} \to 0$
is exact by Corollary \ref{noptorsionabelian}.\\
Thanks to
  A. Chambert-Loir and L. Moret-Bailly for pointing out these arguments to us.
\end{remark}

\medskip

We finish this section with a summary in the case of semiabelian varieties over $K$ ($\omega_1$-saturated) 
$$0 \longrightarrow T \longrightarrow G \longrightarrow E
\longrightarrow 0,$$
with $E$ an elliptic curve.








\begin{proposition} Let $G$ be as above:\\
(i) If $E$ is supersingular, then 
the $\sharp$ sequence remains exact and $G^{\sharp}$ has relative Morley rank.\\ 
(ii) If $E$ is ordinary   and does not descend to the constants
then 
the $\sharp$ sequence remains exact and $G^{\sharp}$ has relative Morley  rank .\\
(iii) If $E$ is ordinary and descends to the constants,  the following
are equivalent 

-- the $\sharp$ sequence  is  exact

-- $G$ descends to the constants 

-- $G^{\sharp}$ has relative Morley rank 

-- $G[p^\infty](K)$ is infinite. 
\smallskip 

\noindent In the case when $G$ does not descend to the constants, then $(T(K) \cap G^\sharp ) / T^\sharp $ is isomorphic to the profinite group $\Z_p$.  
\end{proposition}
\begin{proof} Recall first that Proposition \ref{Morleysemiabelian} says that in the present context
   $G^\sharp$ has relative Morley rank if and only if the $\sharp$ sequence
  is exact. \\
(i) If  $E$ is supersingular, it has no $p$-torsion and   Corollary
\ref{finitetorsionin$G_3$carp} applies.\\
(ii) If $E$ does not descend to the constants,  Corollary
  \ref{finitetorsionin$G_3$carp} applies.\\
(iii) If $E$ is ordinary and descends to $K^{p^\infty}$, by Proposition \ref{Maindescentcarp}, the
$\sharp$ sequence  will be exact if and only if $G$ descends to
$K^{p^\infty}$. As $T$ has no $p$-torsion, $G[p^\infty](\overline K) \cong
  E[p^\infty] (\overline K)  \cong \Z_{p^\infty}$. 
So if $G$ descends to the constants, then $G[p^\infty] (K)
  = G[p^\infty](\overline K)$ so is infinite. \\
If  $G$ does not  descend to $K^{p^\infty}$, using Proposition
  \ref{abelianseparabletorsion}, there is some $n\ge 1$ such that $G[p^\infty](K)=G[p^n](K)$, hence finite.

 In particular, in this case, 
  $T_p G(K) = \{0\}$. By Proposition \ref{tate}, 
$ (T(K) \cap G^\sharp ) / T^\sharp $ is isomorphic to $ T_p E(K)/ \tilde f^{\omega} (T_pG(K))
  \cong T_p E(K) \cong \Z_p$, completing the proof of (iii).  
 \end{proof}

\section{Uniform results in all characteristics}

\vlabel{sectionuniform}

In order to prove the analogues of Proposition \ref{Maindescentcarp} and
Corollary \ref{nonexactexamplecarp} in characteristic $0$, we need to use more differential
  algebraic methods, and in particular D-structures.  But in fact the elementary proofs we gave in the
  previous section for the characteristic $p$ case can also be seen as
  involving D-structures and being  similar to the characteristic $0$
  case. This was just ``hidden'' by the fact that the objects manipulated
  have, in characteristic $p$,  an easy algebraic description. 
We believe it is interesting though
  to explain precisely this uniformity and in order to do this we will
  need to introduce D-structures on group schemes.


But before we launch into this slightly dry exposition, let us point out
that most of the ``uniform'' results can in fact be read at the  ``analogy''
level, without actually understanding the $D$-structure in the
characteristic $p$ case. \\
This will be briefly  explained at the  beginning of Section \ref{Gtildecharp}.

For the whole of this section, $(K,\partial)$ will be a model of $DCF_0$ or $CHF_{p,1}$, where in the latter case we assume $\omega_1$-saturation.

\subsection{D-structures and descent}\vlabel{sectionDstructures}

A good exposition of notions presented here can be found in
\cite{KowalskiPillay2}, one can also look at \cite{Benoist3}. 

\begin{definition}
\begin{enumerate}
\item
An (iterative) Hasse D-structure on a scheme $X$ over $K$ is an (iterative) Hasse derivation $\partial$ on the structure sheaf $\mathcal O_X$, which means that for each open subset $U\subset X$, we have an (iterative)
Hasse derivation $\partial_U:\mathcal O_X(U) \to \mathcal O_X(U)$, such that the structure homomorphism $K\to \mathcal O_X(U)$ and the restriction homomorphisms $\mathcal O_X(U') \to \mathcal O_X(U)$ preserve the Hasse derivations.
\item A morphism of schemes with (iterative) D-structure $(X,\partial_X) \to (Y,\partial_Y)$ is a morphism of schemes $X\to Y$ such that the corresponding morphism of sheaves preserves the Hasse derivations.
\item In particular, for $(R,\partial)$ an (iterative) Hasse differential algebra over $(K,\partial)$, $(\text{Spec}(R),\partial)$ is a scheme with an (iterative) D-structure, and a D-point of $(X,\partial_X)$ with value in $R$ is by definition a morphism
of schemes with (iterative) D-structure $(\text{Spec}(R),\partial) \to (X,\partial_X)$. We denote this set of D-points by $(X,\partial_X)^{\partial}(R)$.
\item If $(X,\partial_X)$ is a scheme with an (iterative) D-structure and $Y$ a closed subscheme of $X$, we say that $Y$ is an (iterative) D-subscheme of $(X,\partial_X)$ if $\partial_X$ induces an (iterative) Hasse derivation on $\mathcal O_Y$, or
equivalently, if the sheaf of ideals $\mathcal I_Y \subset \mathcal O_X$ is a sheaf of D-ideals (i.e. for each open subset $U\subset X$, $\mathcal I_Y(U)$ is an ideal of $\mathcal O_X(U)$ stable by $\partial_U$). 
\item We say that $(G,\partial)$ is a group scheme with an (iterative) Hasse derivation if $G$ is a group scheme over $K$, with an (iterative) D-structure $\partial$, such that the identity element is a D-point 
with value in $K$, and such that the inverse map and the multiplication map
are morphisms of schemes over $K$ with (iterative) D-structure.
\end{enumerate}
\end{definition}

\begin{remark}
For this last point, we have used the fact that if $(X,\partial_X)$ and $(Y,\partial_Y)$ are schemes with an (iterative) Hasse derivation over $K$, $X\times_K Y$ can be endowed in a unique way with an (iterative) Hasse derivation such that
the projection maps are morphisms in this category. It is a straightforward consequence of the existence of tensor products in the category of (iterative) Hasse differential algebras over $K$.
\end{remark}

In the case of algebraic groups over $K$, we can give another description of (iterative) D-structures, which uses the notion of prolongations. The two approaches coincide, see \cite{Benoist}
or \cite{KowalskiPillay2}.

We first recall the description of the prolongations for Hasse derivations, given in the greatest generality in \cite{MoosaScanlon} (see also
\cite{Buium2} or \cite{Vojta}).\\
If $V$ is a smooth irreducible algebraic variety over $K$, 
the $n$-th prolongation of $V$ is an algebraic variety $\Delta_n V$ over $K$ characterized as follows. For any $K$-algebra 
$\phi:K\to R$, the set of $R$-points of $\Delta_n V$ is
$ \Delta_nV(R)= V(R^{(n)})$, where $R^{(n)}=R[X]/(X^{n+1})$ is endowed with the structure of a $K$-algebra by the Taylor map $K\to R^{(n)}$, $a\mapsto \sum_{i=0}^n \phi(\partial_i(a))X^i$.\\
For example, if $V\subseteq \A^n$ is a smooth irreducible affine variety, then $\Delta_n V$ can be described as  
the Zariski-closure 
of the image
of $V(K)$ by $\partial_{\leq n}:=(\partial_0, \ldots,\partial_n)$,
$$
\Delta_nV:=\overline{\{\partial_{\leq n}(x) \, : \, x\in V(K)\}} \subseteq \A^{m{n+1}}.
$$ 
In general, using the Taylor map $K\to K^{(n)}$, we get a (definable) map 
$\partial_{\leq n}:V(K) \longrightarrow \Delta_n V(K)$, having Zariski-dense image. For $m\ge n \ge 0$, we have a 
natural projection morphism $\pi_{m,n} : \Delta_m V \longrightarrow \Delta_nV$ such that 
$\pi_{m,n}\circ \partial_{\le m} = \partial_{\le n}$.\\
These constructions are functorial, and in the
case where $V=G$ is a connected algebraic group, each $\Delta_n G$ has a natural structure of an algebraic group and the maps
$\partial_{\le n}$, $\pi_{m,n}$ are homomorphisms.

\begin{fact} \vlabel{ftDstructure}
Let $G$ be a connected algebraic group over $K$. 
There is a bijective correspondence between the
D-structures on the group scheme $G$ and the sequences of homomorphic regular sections $s=(s_n)_{n\in \N}$ for the projective
system $(\pi_{m,n} : \Delta_m G \longrightarrow \Delta_n G)_{m\ge n \ge 0}$ (i.e. we require that each $s_n : G \longrightarrow \Delta_n G$ is a regular homomorphism over $K$, and 
that these homomorphisms satisfy $\pi_{m,n} \circ s_m = s_n$ and $s_0=\text{id}_G$).\\
The condition for a D-structure to be iterative translates into obvious, but laborious to write out, conditions on the corresponding sequence of sections (see \cite{Benoist} or \cite{KowalskiPillay2}).\\
For $(G,\partial)$ a connected algebraic group with a D-structure $s$ over $K$, the corresponding sequence of sections and $(L,\partial)$ an extension of $K$, the set of D-points can be described as
the \infdef subgroup of $G(L)$:
$$
(G,\partial)^{\partial}(L)=\{x\in G(L) \, : \, \partial_{\le n}(x)=s_n(x) \text{ for all } n\ge 0\}.
$$ 
Moreover, if $L$ is a model ($\omega_1$-saturated in characteristic $p$), then $(G,\partial)^{\partial}(L)$ is Zariski-dense in $G$, and has transcendence degree equal to the dimension of $G$.
\end{fact}

\begin{remark} \vlabel{rationthin}
Let $G$ be a semiabelian variety over $K$. In order to define an iterative D-structure on $G$, it suffices that, for some (any) generic point $g$ of $G^{\sharp}(L)$ over $K$ ($L$ an elementary extension of $K$),
for any $n\ge 0$, $\partial_n(g)\in K(g)$. Indeed, because $G^{\sharp}$ is Zariski-dense in $G$, the existence of such a point $g$ induces a rational map from $G(L)$ to $\Delta_n G (L)$, which can be extended to a homomorphism
$s_n$ by a classical  stability argument. We obtain in this way a D-structure on $G$ because $s_n$ coincides with $\partial_{\le n}$ on the Zariski-dense subgroup $G^{\sharp}$, and the $\partial_{\le n}$'s give a sequence
of definable sections by definition. The iterativity comes from the iterativity of $\partial$, because on an affine open subset $U$, the Hasse derivation given by $(s_n)$ is such that $(Frac(\mathcal O_G(U)), D)$ is isomorphic to
$(K(g),\partial)$, which is an iterative Hasse field.\\
In particular, if $G$ is defined over the constants $\C$, for each $g\in
G^{\sharp}=G(\C)$, $\partial_n(g)=0$ for $n\ge 1$, hence we can define a
natural iterative D-structure on $G$. The two following results are a converse of  this observation.  
\end{remark}

\begin{fact}\vlabel{UnicityDstructure}
Let $G$ be a connected algebraic group over $K$.
Then for each $n\ge 0$, the kernel of $\pi_{n,0}: \Delta_nG \longrightarrow G$ is a unipotent group (see \cite{Pillay-countable} in characteristic $0$ or \cite{Benoist} in arbitrary characteristic). It follows that a semiabelian variety 
$G$ over $K$ admits at most one D-structure, since the
difference between two sections is a homomorphism $G\longrightarrow \text{Ker}(\pi_{n,0})$, hence zero.
\end{fact}

\begin{proposition}\vlabel{Dstructuredescent}
Let $G$ be a semiabelian variety over $K$ with an iterative D-structure. Then $G$ descends to the constants.
\end{proposition} 
\begin{proof}
In the characteristic $0$ case, this result appears implicitly in \cite{Buium1}, see Lemma 3.4 in \cite{BePi} for more explanations.\\
In the characteristic $p$ case, it is proved in \cite{BeDe} (Proof of
Theorem 4.4),   that such a semiabelian variety $G$ descends to $K^{p^n}$ for every $n$.
Then Proposition \ref{modulispace} applies.\end{proof}

\medskip

Note that in characteristic $0$, since an iterative Hasse derivation satisfies $\partial_i=\frac{1}{i!}\partial_1$, it suffices to have a usual derivation $D_1$ on $\mathcal O_G$, or equivalently a section $s=s_1:G \longrightarrow \Delta_1 G$ in 
order to define an iterative D-structure; $\Delta_1 G$ is also known as the twisted tangent bundle of $G$.

\bigskip 
We will now state the criteria for descent which we will be using. 
\medskip 

In the characteristic zero case  we quote from \cite{BePi},
section 3.1. and, in characteristic $p$,  this is the object of section \ref{Gtildecharp}.

\medskip 

\noindent (Characteristic $0$) Let $G$ be a semiabelian variety over $K$ and let $\tilde{G}$ denote the
universal extension of $G$ by a vector group (as defined in
\cite{Rosenlicht2}). Let us write  
$\tilde G$
as $$0 \longrightarrow W_G \longrightarrow \tilde G \longrightarrow G
\longrightarrow 0$$ where $W_G$ is a vector group.

\begin{fact} \vlabel{Dunipotent} 1. $\tilde{G}$ admits a unique iterative D-structure. 

\noindent 2. Consider $U_G$ the maximal subgroup of $W_G$ which is a D-subgroup of $(\tilde G,\partial)$. We still denote by $\partial$ 
the D-structure induced on $\tilde G/U_G$. Then
$G^{\sharp}$ is isomorphic to $(\tilde G/U_G ,\partial)^{\partial}(K)$. 

\noindent 3. It follows
from Proposition \ref{Dstructuredescent} that $G$ descends to the
constants if and only if $G\simeq \tilde{G}/W_G$ has a D-structure if
and only if $U_G=W_G$ (since in this case the projection $\tilde G \to \tilde G/W_G$ must preserve the D-structures, see Corollary 3.6 in \cite{BePi}).

\noindent 4. Furthermore the  functor of D-points is exact on
  the class of algebraic D-groups (\cite{kowalski-pillay}). In
  particular,  ${(\tilde
  G/U_G,\partial)}^{\partial}(K) \cong {(\tilde   G, \partial)^{\partial}}(K) / {(U_G ,\partial)^{\partial}}(K)$.
\end{fact}


\subsection{D-structure on $\tilde G $ in characteristic $p$} \label{Gtildecharp}

In this section, $char(K)=p$.
\medskip 

 In characteristic $p$, the universal extension of $G$ by
a vector group {\em does not in general have an iterative
D-structure}. Indeed, if $G$ is an arbitrary semiabelian variety, 
$(H,D)$ any connected algebraic
$K$-group with an iterative D-structure and $f$ a separable  morphism from $H$
onto $G$, then $G$ must be isogenous to  $(G_0)_K$ for some
$G_0 $ semiabelian variety over the constants: 
$f$ maps $(H,D)^{\partial}(L)$ onto $G^{\sharp}(L)$ by density ($L$ a
sufficiently saturated extension of $K$), and it follows that $K(\{g\})$ is finitely
generated over $K$ as a field where $g$ is a generic point of
$G^{\sharp}(L)$, which implies the conclusion by an argument given in
\cite{BeDe} (compare with Remark \ref{rationthin} and Proposition
\ref{Dstructuredescent}). 
This explains why the
introduction of group schemes (or proalgebraic groups) with D-structures
will be unavoidable in a uniform treatment of both characteristics.

The construction we describe below is, as we mentioned before, the
 D-structure argument which lies  behind the simple algebraic
 treatment we gave in Section  \ref{exactnessdescent}. But, as also
 mentioned at the beginning of Section \ref{sectionuniform}, 
most of the ``uniform'' results can in fact be read at the  ``analogy''
level, without actually understanding the $D$-structure in the
characteristic $p$ case. 

\smallskip 
\noindent More precisely:

Recall from section \ref{charp} that if  $G$ is a semiabelian variety
over $K$, 
$$\tilde{G}:=\mathop{\text{lim}}_{\leftarrow} (G_0\stackrel{[p]}{\longleftarrow} G_1 \stackrel{[p]}{\longleftarrow} \ldots ),$$
 Denote $T_pG$ by $W_G$, and $T_p G(K)$ by $(U_G,\partial)^\partial (K)$ and $\tilde G (K)$ by $(\tilde
G,\partial)^\partial (K)$. These are $*$-definable groups in $K$. 
From section  \ref{charp}, we know that \begin{center}$ G^{\sharp} \mbox{ is
  isomorphic to }  (\tilde G ,\partial)^\partial(K) /{(U_G ,
  \partial)^\partial (K)}  $.\end{center}

Then one can more or less jump to Section \ref{unifstat} and read the 
statements and proofs of Lemma \ref{exactnessofGtilde} 
and  Proposition \ref{exactnessofU}, as they are,  with the above definitions for the
characteristic $p$ case. Except for
condition (iv) in Proposition \ref{exactnessofU} , which then  makes sense only in
characteristic $0$.  

\medskip  

\noindent {\em We will now begin the real construction:} 

\smallskip

 We have $G$   and $\tilde G$ as above.  The (scheme-theoretic) kernel of the
projection $\pi:\tilde{G} \to G_0$ is the Tate module $T_pG$. From now on, it is important to consider $\tilde G$ and $T_pG$ as group schemes. 
We will denote by $X_0$ a system of coordinates for $G_0$ (for a fixed affine covering say), such that $X_0$ generates
the maximal ideal of the identity element of $G_0$, and by $X_i$
its image in $\mathcal O_{G_i}$ by the identity isomorphism.
It follows
that the sections of the sheaf 
$$
\mathcal O_{\tilde{G}}=\mathop{\text{lim}}_{\rightarrow} (\mathcal O_{G_0}  \stackrel{[p]^*}{\rightarrow} \mathcal O_{G_1} \ldots)
$$ 
are generated over $K$ by $(X_i)_{i\in \N}$.\\
Here are some definitions which will play a role in the construction of the D-structure on $\tilde{G}$ in the proof of Proposition \ref{DstructureGtilde}.
The map $[p]^*$ induced by $[p]$ on $\mathcal O_{\tilde G}$ is given by $[p]^*(X_0)=[p]^*_{G_0}(X_0)$ and $[p]^*(X_i)=X_{i-1}$ for $i\ge 1$. We can define a ``shift'' homomorphism $s$ on $\tilde G$ characterized by $s^*(X_i)=X_{i+1}$. It is clear that
$s\circ [p]=id_{\tilde G}$, and that $s$ and $[p]$ commute.\\
At the level of points, for $(g_0,g_1,g_2,\ldots)\in \tilde G(L)$, $L$ an extension of $K$, $s(g_0,g_1,g_2,\ldots)=(g_1,g_2,g_3,\ldots)$ and $[p](g_0,g_1,g_2,\ldots)=([p]g_0,g_0,g_1,\ldots)$.\\ 
Now for each $n$, we have $[p^n]_{\tilde G}=V_n\circ F^n$, where $F^n:\tilde G \to \tilde G ^{(p^n)}$ is the power of the Frobenius homomorphism and $V_n : \tilde{G}^{(p^n)} \to \tilde G$ the $n$-th Verschiebung (induced by the Verschiebung on each $G_ i$).\\

\begin{proposition} \label{DstructureGtilde}
There exists an iterative D-structure on $\tilde G$. Moreover, this D-structure is unique ``in a strong sense''.
\end{proposition}
\begin{proof}
We first state the uniqueness in a strong sense: for any homomorphism of K-algebras $D_0:\Oc_{\tilde G} \to A$ (strictly speaking we should replace $\Oc_{\tilde G}$ by its ring of sections for some open set), there is at most
one (non iterative) Hasse derivation from $\Oc_{\tilde G}$ to $A$ over $K$ extending $D_0$. By this we only mean a sequence of additive maps $(D_i:\Oc_{\tilde G} \to A)_{i\in \N}$ satisfying the generalized Leibniz rule and agreeing with $\partial$ on $K$ 
(we can not require iterativity at this level of generality since $A\neq \mathcal O_{\tilde G}$).\\
We assume that we have such a Hasse derivation $(D_i)_{i\in \N}$, and we consider some $f \in \mathcal O_{\tilde G}$ and some index $i<p^n$. Because of the previous equalities, we must have $D_i(f)=D_i(F^{n*}\circ V_n ^* \circ s^{n*} (f))$. 
But $F^{n*}(V_n^* \circ s^{n*}(f))$ can be represented locally as a rational function
of the variables $X_j^{p^n}$, hence there is a unique possible value for $D_i(f)$ because if $P$ is a polynomial with coefficients in $K$, $D_i(P(X^{p^n}))=P^{\partial_i}(D_0(X)^{p^n})$ where $P^{\partial_i}$ is obtained by 
applying $\partial_i$ to the coefficients of 
$P$.\\
Now we start to define a truncated Hasse derivation on $\Oc_{\tilde G}$. Since $\tilde G^{(p^n)}$ descends to $K^{p^n}$, on which $\partial_{< p^n}$ is trivial, we obtaine a truncated Hasse derivation 
$D'_{<p^n}$ on $\mathcal O_{\tilde G^{(p^n)}}=
\mathcal O_{G'} \otimes_{K^{p^n}} K$ ($G'$ is a model of $\tilde G^{(p^n)}$ over $K^{p^n}$) by putting the trivial truncated Hasse derivation on $\mathcal O_{G'}$. Now we define
$$
D_{<p^n}=F^{n*}\circ D'_{<p^n} \circ V_n^* \circ s^{n*}.
$$   
It is clear that $D'_{<p^n}$ preserves the comultiplication and coinverse of $G'$ because these are $K^{p^n}$-morphisms, hence $D_{<p^n}$ preserves the group structure of $\tilde G$. And because of the uniqueness that we have noticed before, the $D_{<p^n}$ for
different $n$'s are compatible, hence we have defined a (the unique) D-structure D on $\tilde G$. It is actually iterative since the $D'_{<p^n}$ are (as tensor product of the trivial iterative derivation and $\partial$) and since
$V_n^*\circ s^{n*} \circ F^{n*}=id$ (because $F^n\circ s^n \circ V_n=F^n(s^n) \circ F^n \circ V_n = F^n(s^n\circ [p^n]_{\tilde G})=id$).
\end{proof}
\begin{remark}
Let us give a slightly informal description of this D-structure in term of sections $s_n :\tilde G \to \Delta_n \tilde G$. For instance, $s_1 : \tilde G \to \Delta_1 \tilde G = \stackrel{\text{lim}}\leftarrow \Delta_1 G_i$ is given
by the sequence of homomorphisms $(s_{1i})_i$:

~\\
\begin{center}
\begin{pspicture}(5,2)
\rput(0,0){\rnode{a00}{$\Delta_1 G_0$}}
\rput(2,0){\rnode{a20}{$\Delta_1 G_1$}}
\rput(4,0){\rnode{a40}{$\Delta_1 G_2$}}
\rput(5,0){\rnode{a}{$\ldots$}}
\rput(0,2){\rnode{a02}{$G_0$}}
\rput(2,2){\rnode{a22}{$G_1$}}
\rput(4,2){\rnode{a42}{$G_2$}}
\rput(5,2){\rnode{b}{$\ldots$}}
\psset{arrows=->,nodesep=3pt,shortput=tablr,linewidth=0.1pt}
\ncline{a20}{a00}_{$\Delta_1[p]$}
\ncline{a40}{a20}_{$\Delta_1[p]$}
\ncline{a22}{a02}^{$[p]$}
\ncline{a42}{a22}^{$[p]$} 
\ncline{a00}{a02}<{$\pi_0$}
\ncline{a20}{a22}<{$\pi_1$}
\ncline{a40}{a42}<{$\pi_2$}
\ncline{a22}{a00}<{$s_{11}$}
\ncline{a42}{a20}<{$s_{12}$} 
\end{pspicture}
\end{center}
~\\
where, if $x_i \in G_i(K)$ ($i\le 1$), 
$$
s_{1i}(x_i)=([p]x_i,V^{\partial_1}\circ F(x_i))\in \Delta_1G_{i-1}(K).
$$
As in the proof of Propostion \ref{DstructureGtilde}, $V^{\partial_1}$ is defined so that for $K$-rational points $V^{\partial_1}(F(x))=\partial_1(V(F(x)))$ (recall that $F(x)$ is a constant for $\partial_1$), which corresponds to applying
$\partial_1$ to the coefficients
of $V$ when $V$ is a polynomial, with the obvious generalization for rational functions. Let us remark that for every $(a_i)_{i \in \N} \in \tilde G (K)$, $((a_i)_i,(\partial_1(a_i))_i)=s_{1}((a_i)_i)$; it is actually a general fact, proved in Lemma \ref{Dpoints}.

\end{remark}

\begin{remark} \vlabel{strongunique}
It follows from the uniqueness in the strong sense that if $(X,\partial_X)$ is a scheme with a D-structure, and $f:X\to \tilde G$ a morphism of schemes, it is automatically a morphism of schemes with a D-structure: for any open subset $U\subseteq \tilde G$,
the corresponding homomorphism $f^*_U : \mathcal O_{\tilde G}(U) \to \mathcal O_X(f^{-1}(U))$ is such that $\partial_X \circ f^*_U$ and $f^*_U \circ \partial_{\tilde G}$ are two Hasse derivations extending $f^*_U$, hence must coincide, which means that
$f^*_U$ is a D-homomorphism.
\end{remark}

We now focus on the D-points of $(\tilde G, D)$. 
\begin{lemma} \vlabel{Dpoints}
Let $(R,\partial)$ be an iterative Hasse differential $K$-algebra. Then
$(\tilde G,D)^{\partial}(R)=\tilde G (R)$. Of course it is still true for every D-subscheme of $(\tilde G,D)$.
\end{lemma}
\begin{proof}
It is simply Remark \ref{strongunique} for the particular case $(X,\partial_X)=(Spec\;R,\partial)$.
\end{proof}

In Fact \ref{Dunipotent} we defined, in characteristic $0$,
$U_G$ to be the maximal subgroup of $W_G$ which is a D-subgroup of $(\tilde G,D)$.  
Here is the characteristic $p$ version:

\begin{definition} \label{UGcharp}
We define $U_G$ as the maximal closed subscheme of $W_G:= T_pG$ which is a D-subscheme of $(\tilde G,D)$. If we have chosen $X_0$ such that it generates the maximal ideal of the identity element
of $G_0$, $U_G$ is the D-subscheme of $(\tilde G,D)$ defined by the sheaf of D-ideals $\mathcal I_U$ of $(\mathcal O_{\tilde G},D)$ generated by $X_0$. We see that $U_G$ is actually a group D-subscheme of $(\tilde G,D)$ because D preserves the group structure,
which implies that $\iota(\mathcal I_U)\subseteq \mathcal I_U$ and $\mu (\mathcal I_U)\subseteq \mathcal I_U \otimes \mathcal I_U$.
\end{definition}

As in characteristic $0$, we obtain: 

\begin{lemma} \vlabel{UGdesc}
$G$ descends to the constants if and only if $T_pG(= W_G) =U_G$.
\end{lemma}
\begin{proof}
The argument is standard. We know that $G$ descends to the constants if
and only if it admits an iterative D-structure (see section
\ref{sectionDstructures}). If $T_pG=U_G$, $T_pG$ is a D-subgroup of $\tilde G$, hence there is an iterative D-structure on the quotient $\tilde G/T_pG \simeq G$ (it is done in the characteristic $0$ case in \cite{kowalski-pillay}, details in characteristic $p$ are worked out in \cite{Benoist3}).
For the converse, if $G$ descends to the constants, then up to isomorphism, each $G_i$ is endowed with the trivial iterative D-structure, and each $[p]$ map is a D-morphism. It follows that the unique (iterative) D-structure on $\tilde G$
is the trivial one, for which $T_pG$ is a D-subgroup of $\tilde G$, hence equal to $U_G$.
\end{proof}

\begin{lemma} \vlabel{UG}
$U_G(K)=(U_G,D)^{\partial}(K)=T_pG(K)$.
\end{lemma}
\begin{proof}
By Lemma \ref{Dpoints}, each point in $T_pG(K)$ is a D-point of $(\tilde G,D)$. It follows that the corresponding closed point of $T_pG$ is a maximal D-ideal of $(\mathcal O_{\tilde G},D)$ containing $X_0$, hence it is in $U_G$. Conversely, we 
have $U_G(K)\subseteq T_pG(K)$.
\end{proof}

\begin{lemma} \vlabel{Kbar}
$U_G(K)=U_G(\overline K)$
\end{lemma}
\begin{proof}
Let $\mathcal{I}_U$ be the sheaf of D-ideals defining $U_G$ (in fact we consider its sections on an affine open set). The reduced scheme $(U_G)_{red}$ is defined by $\sqrt{\mathcal I_U}$. It is well known that $\sqrt{\mathcal I_U}$ is the intersection
of all the prime ideals containing $\mathcal I_U$. But, since $\mathcal I_U$ is a differential ideal, $\sqrt{\mathcal I_U}$ is also the intersection of all the prime D-ideals containing $\mathcal I_U$ (see \cite{Benoist} for example).\\
Now consider $M$, any maximal ideal of $\Oc_{\tilde G}$ containing $\mathcal I_U$. We want to show that $M$ is a D-ideal. Let $f$ be in $M$, it is in $\Oc_{G_i}\subseteq \Oc_{\tilde G}$ for some $i$. Let $j<p^n$ be some index. From the first remark,
$\bigcap P \subseteq M$, where $P$ runs over the prime D-ideals of $\Oc_{\tilde G}$ containing $\mathcal I_U$. In particular, $\bigcap (P\cap \Oc_{G_{i+n}}) \subseteq M\cap \Oc_{G_{i+n}}$, but the first one 
is the ideal defining $(\pi_{i+n}(U_G))_{red}$, a finite 
scheme (here $\pi_{i+n}$ is the projection of $\tilde G$ onto $G_{i+n}$). It follows that  $M\cap \Oc_{G_{i+n}} =P\cap \Oc_{G_{i+n}}$ for some D-ideal $P$ containing $\mathcal I_U$ (which may depend on $i$ and $n$). But now we have
$D_j(f)=F^{n*}\circ D'_j\circ V_n^* \circ s^{n*}(f)$, with $s^{n*}(f)\in \Oc_{G_{i+n}}$ by definition and $D_j(f)$ as well because $\Oc_{G_{i+n}}$ is stable under $F^{n*}\circ D'_j \circ V_n^{*}$. As $P$ is a D-ideal, 
we have $D_j(f)\in (P\cap \Oc_{G_{i+n}})\subseteq M$.That is, $M$ is a D-ideal.\\
Now we can conclude: for any point $x\in U_G(\overline K)$, the corresponding maximal ideal $M$ of $\Oc_{\tilde G}$ is a D-ideal, hence the residue field $K(x)$ is an algebraic D-extension of $K$. But we know that any algebraic D-extension of $K$ is trivial because
$K$ is existentially closed
(see \cite{Ziegler1} for example), hence $x\in U_G(K)$.
\end{proof}\\

In order to deal with a big chunk of the non reduced part of $T_pG$, we introduce the following morphism of group schemes $\tilde F: \tilde G \to \tilde{\tilde G}$, for $\tilde F$ and $\tilde{\tilde G}$ defined as follows (recall that $F^i$
is the Frobenius homomorphism $G\to G^{(p^i)}$):

~\\
\begin{center}
\begin{pspicture}(7,2)
\rput(0,2){\rnode{a02}{$\tilde G$}}
\rput(0.5,2){\rnode{a}{=}}
\rput(1,2){\rnode{a12}{${\stackrel{\text{lim}}{\leftarrow}}$}}
\rput(2,2){\rnode{a32}{$G_0$}}
\rput(4,2){\rnode{a52}{$G_1$}}
\rput(6,2){\rnode{a72}{$G_2$}}
\rput(7,2){\rnode{a82}{$\ldots$}}
\rput(0,0){\rnode{a00}{$\tilde{\tilde G}$}}
\rput(0.5,0){\rnode{b}{=}}
\rput(1,0){\rnode{a10}{${\stackrel{\text{lim}}{\leftarrow}}$}}
\rput(2,0){\rnode{a30}{$G_0$}}
\rput(4,0){\rnode{a50}{$G_1^{(p)}$}}
\rput(6,0){\rnode{a70}{$G_2^{(p^2)}$}}
\rput(7,0){\rnode{a80}{$\ldots$}}
\psset{arrows=->,nodesep=3pt,shortput=tablr,linewidth=0.1pt}
\ncline{a52}{a32}^{$[p]$}
\ncline{a72}{a52}^{$[p]$}
\ncline{a50}{a30}_{$V_G$}
\ncline{a70}{a50}_{$V_{G^{(p)}}$} 
\ncline{a02}{a00}<{$\tilde F$}
\ncline{a32}{a30}<{id}
\ncline{a52}{a50}<{$F$}
\ncline{a72}{a70}<{$F^2$} 
\end{pspicture}
\end{center}
~\\

\noindent
This is well-defined since for all $i$, $F^i\circ[p]=[p]\circ F^i=V_{G^{(p^i)}}\circ F^{i+1}$.\\ 
For each $i$, we will denote by $Y_i$ a system of coordinates of $G_i^{(p^i)}$ such that $F^{i*}(Y_i)=X_i^{p^i}$; $Y_i$ generates the maximal ideal of the identity element in $G_i^{(p^i)}$.
\begin{lemma}
$Ker(\tilde F)$ is a D-subscheme of $\tilde G$. In particular $Ker(\tilde F)\subseteq U_G$.
\end{lemma}
\begin{proof}
The sheaf of ideals $\mathcal I$ of the kernel is generated by the $F^{i*}(Y_i)$'s, that is by the $X_i^{p^i}$'s. If $p^i$ does not divide $j$, $D_j(X_i^{p^i})=0$; and if $j=hp^i$ for some $h<p^n$, we have $D_j(X_i^{p^i})=(D_h(X_i))^{p^i}$,
with
$$
D_h(X_i)=F^{n*} \circ D'_h \circ V_n^* (X_{i+n}).
$$
Here the map $V_n:\tilde G^{(p^n)}\to \tilde G$ comes from the homomorphisms $V_n:G_k^{(p^n)} \to G_k$, $k\ge 0$. It follows that $V_n^*(X_{i+n})\in \mathcal{M}$, the maximal ideal of the identity element in $\mathcal O_{G_{i+n}^{(p^n)}}$. Since
the identity element of
$G_{i+n}^{(p^n)}$ is a $\text{D}_{<p^n}$-point for the trivial truncated D-structure on $G_{i+n}^{(p^n)}$, $D'_h(\mathcal M)\subseteq \mathcal M$. Hence $D_h(X_i)\in (F^{n*}(\mathcal M)) \subseteq (X_{i+n}^{p^n})$, and $D_j(X_i^{p^i})\in 
(X_{i+n}^{p^{i+n}}) \subseteq \mathcal I$.\end{proof}\\

We now need $G$ to be ordinary. We will use here the structure theorems for affine commutative group schemes over a field $L$ (note that $Ker(\tilde F)$, $U_G$, $T_pG$ are objects of this category since they are commutative profinite), for which our reference is \cite{DG} (chap. III, \textsection 3 and chap. V, \textsection 3). Let us recall that if $H$ is a profinite commutative group scheme over $L$, the connected component $H^\circ$ of the neutral element is an infinitesimal group subscheme, namely satisfies $H^\circ(\overline L)=\{0\}$. Moreover, the quotient $H/H^\circ$ is pro\'etale (that is a projective limit of finite \'etale groups), and $H^\circ$ is the unique connected group subscheme of $H$ with this property. Furthermore, if we assume $L$ to be algebraically closed, the reduced subscheme $H_{red}$ is isomorphic to $H/H^\circ$ via the projection map (it is in particular a group subscheme of $H$), and $H$ is isomorphic to the direct product $H_{red}\times H^\circ$. Note finally that taking the connected component $H^\circ$ commutes with base change of fields, and the same is true for taking the quotient $H/H^\circ$.

\begin{lemma} \vlabel{UGinf}
Suppose that $G$ is ordinary. Then $(T_pG)^\circ=U_G^\circ=Ker(\tilde F)$.
\end{lemma}
\begin{proof}
We have seen that $Ker(\tilde F)\subseteq U_G \subseteq T_pG$ (as closed group subschemes).
But $T_pG/Ker(\tilde F)$ is a group subscheme of the projective limit
$$
\mathop{\text{lim}}_{\leftarrow} (0\stackrel{V_G}{\longleftarrow} Ker(V_1) \stackrel{V_{G^{(p)}}}{\longleftarrow} Ker (V_2) \ldots ),
$$
which is a pro\'etale group scheme since each $Ker(V_n)$ \'etale ($V_n$ is separable since $G$ is ordinary). 
It follows that $T_pG/Ker(\tilde F)$ and $U_G/Ker(\tilde F)$ are themselves pro\'etale (this category is closed under subobjects). Since obviously $Ker(\tilde F)(\overline K)=\{0\}$, the lemma is proved.
\end{proof}

\medskip

By combining the previous  lemmas, we obtain a new proof of the
Proposition \ref{abelianseparabletorsion} with a ``D-structure
flavor''. Note that for any profinite group scheme $X$  over $K$,
$X_{red}$ is the unique reduced closed subscheme of $X$ such that
$X(\overline K)=X_{red}(\overline K)$ . 

\begin{proposition} \label{U=W} 
Let $G$ be  an ordinary semiabelian variety over $K$. Then $G$ descends to the constants if and only if $T_pG(\overline K)=U_G(\overline K)$ if and only if $T_pG(K)=T_pG(\overline K)$.
\end{proposition}
\begin{proof}
We have obtained that $G$ descends to the constants if and only if $T_pG=U_G$ (\ref{UGdesc}). But $T_pG=U_G$ if and only if $(T_pG)_{\overline K}=(U_G)_{\overline K}$ (since field extensions are obviously faithfully flat), and by the direct sum decomposition, 
$(T_pG)_{\overline K}=(U_G)_{\overline K}$ if and only if $((T_pG)_{\overline K})_{red}=((U_G)_{\overline K})_{red}$ and $(T_pG)_{\overline K}^\circ=(U_G)_{\overline K}^\circ$. But by \ref{UGinf}, $(T_pG)^\circ=U_G^\circ$, hence $T_pG=U_G$ if and only if $T_pG(\overline K)=U_G(\overline K)$. Since
$U_G(\overline K)=U_G(K)=T_pG(K)$ (\ref{UG} and \ref{Kbar}), we have the result.
\end{proof}

~\\

\subsection{Uniform statements and proofs} \vlabel{unifstat}

We will consider $G$ a
semiabelian variety over $K$. Recall that, in characteristic $0$, $\tilde G$ is the
universal extension of $G$  by a vector group and $W_G, U_G$ are defined
as described in Fact \ref{Dunipotent}. In characteristic $p$, $\tilde G$, $U_G$ and $W_G=T_pG$  have been defined in section \ref{Gtildecharp}.


In all characteristics, we know from Fact \ref{Dunipotent} in characteristic $0$, and from Lemmas \ref{sharpTate} and \ref{UG} in characteristic $p$, that 
$$ G^{\sharp} \mbox{ is isomorphic to }  (\tilde G,\partial)^\partial (K) /{(U_G,\partial))^\partial}(K)  $$
where of course, by isomorphic here we mean isomorphic as $*$-definable subgroups.

\noindent Notation: If $f : G \longrightarrow H$ is a morphism of
semiabelian varieties 
over $K$, we denote by $\tilde f$ the induced morphism from $\tilde G$
to $\tilde H$. \\
In the following, we will often drop the $\partial$ for a scheme $(X,\partial)$ with a D-structure when no ambiguity arises.\\
If $H_1$, $H_2$ are proalgebraic groups  over $K$ with a
D-structure, and $h : H_1 \longrightarrow H_2$ is a morphism of
proalgebraic groups with a D-structure, we 
denote by $h^\partial$ the induced $*$-definable homomorphism from
${H_1}^\partial(K)$ to ${H_2}^\partial(K)$.  When $G,H$ are semiabelian
varieties, $\tilde G$ and $\tilde H$ have unique D-structures, and so for any 
$f:G\to H$, $\tilde f$ respects the D-structures (see Remark \ref{strongunique} in characteristic $p$, and Corollary 3.6 from \cite{BePi} in characteristic $0$),
whereby ${\tilde f}^{\partial}$ is defined. 
For the same reason, $\tilde f$ induces the maps $\tilde f_U:U_G \to U_H$ and $(\tilde f_U)^{\partial}:U_G^{\partial}(K) \to U_H^{\partial}(K)$ 
(see sections \ref{sectionDstructures} and \ref{Gtildecharp}).\\

\begin{lemma}\vlabel{exactnessofGtilde} Let  $0\longrightarrow G_1
  \longrightarrow G_2  \stackrel{f}\longrightarrow  G_3\longrightarrow 0$ be an
  exact sequence of  semiabelian varieties over $ K $. Then the
  sequence  $0\longrightarrow (\tilde G_1)^\partial (K)  \longrightarrow
  (\tilde G_2)^\partial (K) \stackrel{{\tilde f}^\partial} \longrightarrow   (\tilde
  G_3)^\partial (K) \longrightarrow 0$ is also exact.  \end{lemma}
\begin{proof} In characteristic $0$, $\tilde G_i$ is the universal
  vectorial extension of $G_i$ and the  sequence
  $$0\longrightarrow \tilde G_1 \longrightarrow \tilde G_2
  \stackrel{\tilde f}\longrightarrow  \tilde G_3\longrightarrow 0$$ is also exact
(see Appendix B).
Each $\tilde G_i$ admits a (unique) D-structure and the functor $H
  \mapsto H^\partial (K)$ preserves exact sequences from the category of
 algebraic groups with a D-structure to the category of definable groups (see \cite{kowalski-pillay}).
In characteristic $p$, the sequence $0\to G_1(K) \to G_2(K) \to G_3(K) \to 0$ is exact because $K$ is separably closed, hence, passing to
the projective limit in the category of $*$-definable groups,  
$$
0\longrightarrow \tilde G_1(K) \longrightarrow \tilde G_2(K)
  \stackrel{\tilde f}\longrightarrow  \tilde G_3(K)
$$ 
is also exact. The fact that $\tilde f : \tilde G_2(K) \to \tilde G_3(K)$ is surjective follows from the surjectivity of $f:G_2(K) \to G_3(K)$ and from the $\omega_1$-saturation of $K$
\end{proof} 

\medskip

The next proposition gives us a very useful equivalent to the
exactness of the $\sharp$ functor. It should be noted that there is
no assumption that any of the $U_{G_i}^\partial$'s, or any of the $U_{G_i}$'s, 
are  non trivial. 

Given the exact  sequence $0\longrightarrow G_1 \longrightarrow G_2
\stackrel{f}\longrightarrow  G_3\longrightarrow 0$, $\tilde f$, $(\tilde f)^\partial$, $\tilde f_U$ and $(\tilde f_U)^{\partial}$ are the induced
maps as above
and $\tilde f_\pi$ denotes the induced map from
$G_2^\sharp$ to $G_3^\sharp$, when we identify 
$G_i^\sharp$ with  $(\tilde G_i)^\partial (K) / (U_{G_i})^\partial (K)$.

\begin{proposition}\vlabel{exactnessofU} Let 
 $0\longrightarrow G_1 \longrightarrow G_2 \stackrel{f}\longrightarrow  G_3\longrightarrow 0$ be an
  exact sequence of  semiabelian varieties over $ K $. 
For convenience, we assume that $G_1$ is a semiabelian subvariety of $G_2$.
Then
  the following are equivalent: 

(i)  $0\longrightarrow {G_1}^\sharp  \longrightarrow {G_2}^\sharp 
  \stackrel{f_\pi}\longrightarrow  {G_3}^\sharp \longrightarrow 0$ is exact \\

(ii) $0\longrightarrow (U_{G_1})^\partial (K) \longrightarrow( U_{G_2})^\partial (K)
  \stackrel{(\tilde f_U)^\partial}\longrightarrow  (U_{G_3})^\partial (K)
  \longrightarrow 0$ is exact\\

(iii) $(\tilde f_U)^\partial \, : (U_{G_2})^\partial (K)\longrightarrow \,
  (U_{G_3})^\partial (K)$ is surjective\\ 

(iv) $0\longrightarrow U_{G_1}(\overline K)  \longrightarrow U_{G_2}(\overline K)
  \stackrel{\tilde f_U}\longrightarrow   U_{G_3}(\overline K) \longrightarrow 0$ is exact. \\

\noindent Furthermore $(G_1(K) \cap {G_2}^\sharp) / {G_1}^\sharp   \, \widetilde{\longrightarrow}
 \,  (U_{G_3})^\partial (K)/ {(\tilde f_U)}^\partial ( ( U_{G_2})^\partial (K))$. 
\end{proposition}

\begin{proof} From
Lemma \ref{exactnessofGtilde}, one obtains the
following commutative diagram of exact sequences (*):

~\\
\begin{center}
\begin{pspicture}(10,5)
\rput(5,0){\rnode{a40}{$0$}}
\rput(0,1){\rnode{a02}{$0$}}
\rput(2,1){\rnode{a22}{$(U_{G_3})^\partial (K)$}}
\rput(5,1){\rnode{a42}{$(\tilde G_3)^\partial (K)$}}
\rput(8,1){\rnode{a62}{$({G_3})^\sharp $}}
\rput(10,1){\rnode{a82}{$0$}}
\rput(0,3){\rnode{a04}{$0$}}
\rput(2,3){\rnode{a24}{$(U_{G_2})^\partial (K)$}}
\rput(5,3){\rnode{a44}{$(\tilde G_2)^\partial (K)$}}
\rput(8,3){\rnode{a64}{$({G_2})^\sharp $}}
\rput(10,3){\rnode{a84}{$0$}}
\rput(5,4){\rnode{a46}{$(\tilde{G_1})^\partial (K)$}}
\rput(5,5){\rnode{a48}{$0$}}

\psset{arrows=->,nodesep=3pt,shortput=tablr,linewidth=0.1pt}
\ncline{a42}{a40}
\ncline{a02}{a22}
\ncline{a22}{a42} 
\ncline{a42}{a62}_{$\pi_3$}
\ncline{a62}{a82}
\ncline{a04}{a24}
\ncline{a24}{a44} 
\ncline{a44}{a64}^{$\pi_2$}
\ncline{a64}{a84}
\ncline{a26}{a46}
\ncline{a48}{a46}
\ncline{a46}{a44}
\ncline{a24}{a22}>{$(\tilde f_U )^\partial$}
\ncline{a44}{a42}>{$(\tilde f)^\partial$}
\ncline{a42}{a40}
\ncline{a64}{a62}>{$\tilde f_{\pi}$}
\end{pspicture}
\end{center}
~\\

\noindent{\bf Claim:}   $Ker ({\tilde f_U})^\partial =
(U_{G_1})^\partial (K)$. 

First note that $(U_{G_i})^\partial (K)=(\tilde{G_i})^\partial (K) \cap W_i$, since it is the kernel of the restriction of $\pi_i$ to $(\tilde{G_i})^\partial (K)$, and 
$W_2 \cap \tilde{G_1}=W_1$, since $W_i$ is
the kernel
of $\pi_i : \tilde{G_i} \to G_i$. It follows that
$(U_{G_1})^\partial(K)=(\tilde{G_1})^\partial (K) \cap W_1 = (\tilde{G_1})^\partial (K) \cap W_2=(\tilde{G_1})^\partial (K) \cap (\tilde{G_2})^\partial (K) \cap W_2 = Ker(({\tilde f})^\partial) \cap (U_{G_2})^\partial (K)=Ker ({\tilde f_U})^\partial$.

\qed 

\medskip 

\noindent Let $S :=(U_{G_3})^\partial (K) / 
(\tilde f_U)^\partial (  (U_{G_2})^\partial (K)) $ (the cokernel of
$(\tilde f_U )^\partial$).  Then the classical Snake Lemma applied to diagram (*) gives the existence
of a homomorphism $d$ from  $Ker (\tilde f_\pi )$ to $S$,   such that the sequence 
$0 \, \longrightarrow (U_{G_1})^\partial \, \longrightarrow \, {(\tilde G_1)}^\partial \, \longrightarrow \,
  Ker (\tilde f_\pi )\, \stackrel{d}\longrightarrow \, S \, \longrightarrow \,
  0 \longrightarrow 0$ is exact 
in the following commutative diagram:

~\\
\begin{center}
\begin{pspicture}(10,5)
\rput(2,0){\rnode{a20}{$S$}}
\rput(5,0){\rnode{a40}{$0$}}
\rput(0,1){\rnode{a02}{$0$}}
\rput(0,4){\rnode{a06}{$0$}}
\rput(2,1){\rnode{a22}{$(U_{G_3})^\partial (K)$}}
\rput(5,1){\rnode{a42}{$(\tilde{G_3})^\partial (K)$}}
\rput(8,1){\rnode{a62}{${G_3}^\sharp $}}
\rput(10,1){\rnode{a82}{$0$}}
\rput(0,3){\rnode{a04}{$0$}}
\rput(2,3){\rnode{a24}{$(U_{G_2})^\partial (K)$}}
\rput(5,3){\rnode{a44}{$(\tilde{G_2})^\partial (K)$}}
\rput(8,3){\rnode{a64}{${G_2}^\sharp $}}
\rput(10,3){\rnode{a84}{$0$}}
\rput(2,4){\rnode{a26}{$(U_{G_1})^\partial (K)$}}
\rput(5,4){\rnode{a46}{$(\tilde{G_1})^\partial (K)$}}
\rput(8,4){\rnode{a66}{Ker$(\tilde f_{\pi}) $}}
\rput(5,5){\rnode{a48}{$0$}}
\rput(8,0){\rnode{a60}{$0$}}
\rput(2,5){\rnode{a28}{$0$}}
\rput(8,5){\rnode{a68}{$0$}}

\psset{arrows=->,nodesep=3pt,shortput=tablr,linewidth=0.1pt}
\ncline{a28}{a26}
\ncline{a68}{a66}
\ncline{a20}{a40}
\ncline{a02}{a22}
\ncline{a06}{a26}
\ncline{a22}{a42} 
\ncline{a42}{a62}^{$\pi_3$}
\ncline{a62}{a82}
\ncline{a04}{a24}
\ncline{a24}{a44} 
\ncline{a44}{a64}^{$\pi_2$}
\ncline{a64}{a84}
\ncline{a26}{a46}
\ncline{a46}{a66}^{$\pi_1$}
\ncline{a26}{a24}
\ncline{a24}{a22}>{$(\tilde f_U)^\partial$}
\ncline{a22}{a20}
\ncline{a48}{a46}
\ncline{a46}{a44}
\ncline{a44}{a42}>{$(\tilde f)^\partial$}
\ncline{a42}{a40}
\ncline{a66}{a64}
\ncline{a64}{a62}>{$\tilde f_{\pi}$}
\ncline{a62}{a60}
\ncline{a40}{a60}

\end{pspicture}
\end{center}
~\\

\noindent This says exactly that \\
$S = (U_{G_3})^\partial (K)/ (\tilde f_U)^\partial
(  (U_{G_2})^\partial (K))$ is isomorphic to $Ker (\tilde f_\pi)  / {\pi_1((\tilde
G_1 )^\partial (K)/ (U_{G_1})^\partial (K))}$, that is, to $(G_1(K) \cap
{G_2}^\sharp )/ {G_1}^\sharp$.
\smallskip 

\noindent It follows in particular that 
 
\noindent $0\longrightarrow {G_1}^\sharp \longrightarrow {G_2}^\sharp 
  \stackrel{f_\pi}\longrightarrow {G_3}^\sharp \longrightarrow 0$ 
is exact\\ if and   only if \\  
$0\longrightarrow (U_{G_1})^\partial  (K) \longrightarrow( U_{G_2})^\partial (K)
  \stackrel{(\tilde f_U)^\partial}\longrightarrow  (U_{G_3})^\partial (K)
  \longrightarrow 0$ is exact\\
if and only if $(\tilde f_U)^\partial$  is surjective. 

\smallskip 

\noindent This is also equivalent to the exactness of the
sequence $0\longrightarrow U_{G_1} (\overline K)  \longrightarrow U_{G_2} (\overline K)
  \stackrel{(\tilde f_U)}\longrightarrow  U_{G_3} (\overline K)
  \longrightarrow 0$. In characteristic $0$, one direction follows from the exactness of the $\partial$
  functor on groups with a D-structure. For the other
  direction suppose that the sequence of the $(U_{G_i})^\partial (K)$'s is
  exact. For each $i$, $U_{G_i}^\partial (K)$ 
  has transcendence degree equal to 
the dimension of the algebraic group $U_{G_i}$ (Fact \ref{ftDstructure}). It follows that 
 $dim U_{G_1} + dim U_{G_3} = dim U_{G_2}$ (by additivity of the transcendence degree).
Being vector groups, the sequence of the $U_{G_i}$'s is exact.
In characteristic $p$, this is a direct consequence of Lemmas \ref{UG} and \ref{Kbar}. 
\end{proof}

\medskip 

We can now give a uniform proof of the main result which relates exactness of
the $\sharp$ functor to questions of descent, restricted, in char. $p$
to the class of {\em ordinary } semiabelian varieties. It is the uniform version of Proposition \ref{Maindescentcarp}.

\begin{proposition}\vlabel{Maindescent}
Let 
$0\rightarrow G_1 \rightarrow G_2 \rightarrow G_3\rightarrow 0$ be an
  exact sequence of (ordinary in char.p) semiabelian varieties defined over $ K $. Suppose
  that $G_1$ and $G_3$ descend to the constants of $K$.

Then the sequence $0\rightarrow G_1^\sharp \rightarrow G_2^\sharp \rightarrow G_3^\sharp \rightarrow 0$ remains exact if and only if  $G_2$ also descends 
  to the constants.  
\end{proposition}

\begin{proof} 
Without loss of generality, we can suppose that $G_1$ is a semiabelian subvariety of $G_2$, and that $G_1=(G_1')_K$
and $G_3=(G_3')_K$, where $G_1'$ and $G_3'$  are semiabelian varieties over ${\cal C}$, the field of constants
of $K$. 

If $G_2$ descends to the constants, then up to isomorphism, we can suppose
that $G_2=(G_2')_K$ for some $G_2'$ over the constants, so for every $i$,
${G_i}^\sharp = G_i' ({\cal C})$. And  then $G_1(K) \cap G_2^\sharp =
G_1(K) \cap G_2 ({\cal C}) = G_1 ({\cal C}) = G_1^\sharp$. 

\noindent For the converse, suppose that $0\rightarrow
{G_1}^\sharp (K) \rightarrow {G_2}^\sharp (K) \rightarrow {G_3}^\sharp (K)
\rightarrow 0$ is exact.  

\noindent By Propostion \ref{exactnessofU}, 
$0\rightarrow U_{G_1} (\overline K) \rightarrow U_{G_2} (\overline K) \rightarrow
U_{G_3} (\overline K)\rightarrow 0$  is also exact. We know that (see Fact \ref{Dunipotent} in characteristic $0$ and Proposition \ref{U=W} in characteristic $p$) as $G_1$ and
$G_3$ descend to the constants, $U_{G_1}(\overline K) = W_1(\overline K)$ and $U_{G_3}(\overline K) = W_3(\overline K)$.
Consider the dimensions, as vector spaces in characteristic $0$ or as free $\Zp$-modules in characteristic $p$, of the $U_{G_i}(\overline K)$'s. By
exactness,  $dim (U_{G_2}(\overline K)) = dim (U_{G_1}(\overline K)) + dim (U_{G_3}(\overline K))$. But we also have
  that $dimW_2(\overline K) = dim W_1 (\overline K)+ dim W_3 (\overline
  K)$ (this follows from exactness of the functor $G\mapsto \tilde G$, which 
   is clear in characteristic $p$, and  proved in Appendix B for
  characteristic $0$). So $dim U_{G_2}(\overline K) = dim W_2(\overline K)$ and  hence $U_{G_2}(\overline K) = W_2 (\overline K)$, 
that is,  again by Fact \ref{Dunipotent} in characteristic $0$ and
Proposition \ref{U=W} in characteristic $p$,  $G_2$ descends to the
  constants. 
\end{proof}


Hence we obtain in arbitrary characteristic the analogue of Corollary \ref{nonexactexamplecarp}, with the same proof.

\begin{corollary}\vlabel{nonexactexample} 
For any ordinary 
abelian variety $A$ defined over the
  constants of $K$,  there exists an exact sequence over $K$, 
$$  0 \longrightarrow {\mathbb G}_m \longrightarrow H \longrightarrow A_K \longrightarrow 0 $$  such that 
$$  0 \longrightarrow {\mathbb G}_m^{\sharp} \longrightarrow H^{\sharp} \longrightarrow (A_K)^{\sharp} \longrightarrow 0$$
is not exact. 
\end{corollary}

We have given some examples of non exactness in characteristic $p$ in Section \ref{further}, even for abelian varieties.
In characteristic $0$, the situation is completely  different for
abelian varieties as shown in the next Proposition, which is a direct consequence of Proposition \ref{exactnessofU}.

\begin{proposition} \vlabel{exactchar0abelian} (Characteristic 0)  Let
  $0 \longrightarrow A \longrightarrow  B \longrightarrow C
  \longrightarrow 0$ be an exact sequence of abelian varieties over $K$.
Then the induced sequence  
$0 \longrightarrow A^\sharp \longrightarrow  B^\sharp \longrightarrow C^\sharp
  \longrightarrow 0$  is also  exact.
\end{proposition}



\begin{proof}  By Poincar\'e  complete reducibility, $A\times C$ is isogenous to $B$, inducing an isogeny of
$\widetilde{A\times C} = \tilde A \times \tilde C$ with $\tilde B$. As this is also an isogeny of D-groups it induces an isogeny between $U_{A\times C} = U_{A}\times U_{C}$ and $U_{B}$.
As these are vector groups it follows that the induced sequence
$0 \longrightarrow U_{A} \longrightarrow  U_{B} \longrightarrow U_{C}
  \longrightarrow 0$  is exact. Hence by Proposition \ref{exactnessofU}, so is
  $0 \longrightarrow A^\sharp \longrightarrow  B^\sharp \longrightarrow C^\sharp
  \longrightarrow 0$. 
\end{proof}


\section{Additional remarks and questions}

\vlabel{final}

1. In characteristic $p$, the counterexamples to exactness of the
induced $\sharp$ sequence arise from the following situation: we have
two connected  commutative definable groups $G_1 <G_2$ which are not
divisible. We consider $D_2$ the biggest divisible subgroup
(which is infinitely definable) of $G_2$. The counterexamples are
exactly the cases when $G_1 \cap D_2$ is {\em not } divisible. One can
ask the same question also for other classes of groups, in particular
for commutative algebraic groups: Given $G_1 <G_2$ two commutative
connected algebraic groups defined over some algebraically closed field
$K$ of characteristic $p$, consider $D <G_2$, the biggest divisible
subgroup of $G_2$. It is easy to check that D is a closed subgroup
of $G_2$, also defined over $K$. 

Using the characterizations of the groups $p^\infty G(K)$, given in
terms of the  Weil restrictions $\Pi_{K/K^{p^n}} G$ in \cite{BeDe}, one
  can deduce easily  from our examples that the same phenomenon occurs for
 commutative algebraic groups.

2. In previous versions of this paper, we had  mentioned an open question which we found 
quite intriguing.
Let $A$ be an abelian variety
   defined over ${\mathbb F}_p(t)$ and let $K_0$ denote the  separable closure
   of ${\mathbb F}_p(t)$. We can consider $A(K_0)$ and $p^\infty A(K_0)$. As we
   recalled in section \ref{sharp},  $p^\infty A(K_0) $ is the
   biggest divisible subgroup of $A(K_0) $ and contains all the
   torsion of $A$ which is prime to $p$. 
The question was whether $p^\infty
   A(K_0)$ could contain any non-torsion element. Note that if
   $A$ is defined over ${K_0}^{p^\infty} = \overline {{\mathbb F}_p}$, then
   $p^\infty A(K_0)= A(\overline {{\mathbb F}_p})$, where indeed every element
   is torsion. 
Note also that, from the beginning of section \ref{sharp}, in
   characteristic $p$, when dealing with $A^\sharp = p^\infty A(K)$, we
   suppose that $K$ is $\omega_1$-saturated, which ensures that
   $A^\sharp$ contains elements which are not torsion.
This question was answered in some particular cases in \cite{Benoist2}, and in full generality by D. R\"ossler who showed in \cite{Rossler} that $p^{\infty}A(K_0)$ contains only torsion points.

In characteristic $0$ there are results along these lines, sometimes
going under the expression ``Manin's theorem of the kernel". A formal
statement and proof (depending on results of Manin, Chai,..) appears
in \cite{BePi} (Corollary K.3 of the Appendix), and says that if $A$
is an abelian variety over the algebraic closure $K_{0}$ say of
${\mathbb C}(t)$, equipped with a derivation with field of constants  $\mathbb 
C$, and $A$ has $\mathbb C$-trace $0$, then $A^{\sharp}(K_{0})$ is precisely the group of torsion points of $A$. This, together with the fact (see \cite{BePi}, section 6 and \cite{MarkerPillay}, Lemma 2.2) that $A^\sharp (K_0) = A^\sharp({K_0}^{diff})$, shows that 
$A^\sharp({K_0}^{diff})$ is the group of torsion points of $A$. 

\section*{Appendix A}


Here is a proof of 

Fact \ref{separablemorphisms} {\em Let $G, H$ be two connected 
algebraic  groups defined over $K$ and $f$ a
dominant separable homomorphism from $G$ to $H$ (equivalently a
surjective separable homomorphism from $G(\overline K)$ to $H(\overline
K)$). Then $f$ takes $G(K)$ surjectively onto $H(K)$.}

\begin{proof} 

Note first that we can suppose without loss of generality that $K$ is
sufficiently saturated. Let $K_0$ be a small field over which everything
is defined. 
Let  $h \in H(K)$ be  a generic point of $H$ over $K_{0}$ (in the sense
of algebraic geometry). As f is dominant, there is some generic $g$ of   
$G(\overline K)$ such that $f(g) = h$. Separability of $f$ means that 
 $K_{0}(g)$ is a separable extension of
 $K_{0}(h)$, hence contained in a separable closure of
$K_{0}(h)(a_{1},..,a_{n})$ 
for some $a_{i}$ which are algebraically independent over
$K_{0}(h)$. Choosing, by saturation of $K$, $b_{1},..,b_{n}\in K$, 
algebraically independent over $K_{0}(h)$, and an isomorphism taking
the separable closure of $K_{0}(h)(a_{1},..,a_{n})$ to 
the separable closure of $K_{0}(h)(b_{1},..,b_{n})$, we find
$g'\in G(K)$ such that $f(g') = h$. 
\end{proof}

\section*{Appendix B }

In this appendix, we give a detailed proof of the fact used in Lemma
\ref{exactnessofGtilde}, namely that the functor "universal extension",
on the category of semiabelians varieties in characteristic $0$, is
exact. As we could not find any references for this fact, maybe well
known, we give the details here, thanks to the help of D. Bertrand. We
refer to \cite{Bertrand} for discussion about related questions. Note
also that the point of view of rigidified extensions used in
\cite{MazurMessing} should give this  result more directly, but we keep here a point of view which model theorists are probably more familiar with.\\
Everything here is over an algebraically closed field $K$ of characteristic $0$, and every algebraic group is commutative.\\

Recall that the universal extension of an algebraic group $A$ by a vector group (when it exists) is an extension
$$ 0 \to W_A \to \tilde A \stackrel{\pi_A}{\to} A \to 0,$$
where $W_A$ is a vector group, characterized by the following universal property:
for every extension $f:G \to A$ of $A$ by a vector group, there exists a unique homomorphism of algebraic groups $g:\tilde A \to G$ 
such that $\pi_A=f\circ g$.\\
~\\
It follows from \cite{Rosenlicht2} that abelian varieties admit such universal extension.
If $S$ is a semiabelian variety, with abelian part $A$, $S$ admits a universal extension by a vector group, which is given by $\tilde S=S\times_A \tilde A$ (see \cite{BePi}); note that $W_S=W_A$.\\
~\\
We should now explain how $\tilde ~$ is defined as a functor on the category of semiabelian varieties.\\
First recall some notations and constructions from \cite{Serrebook}, chap. 7. For algebraic groups $A$ and $B$, $\text{Ext}(A,B)$ is the set of extensions $0 \to B \to C \to A \to 0$ of $A$ by $B$, up to isomorphism of extensions. It is equipped with a structure of a group.\\
~\\
If $C\in \text{Ext}(A,B)$, and $g:B\to B'$, $g_*(C)$ is the unique element $C'\in \text{Ext}(A,B')$ such that there is some $G:C\to C'$ 
such that the following diagram commutes (actually it does not depend only on $C$, but on $C$ as an extension of $A$ by $B$):\\
~\\
\begin{center}
\begin{pspicture}(8,1)
\rput(0,1){\rnode{A}{$0$}}
\rput(2,1){\rnode{B}{$B$}}
\rput(4,1){\rnode{C}{$C$}}
\rput(6,1){\rnode{D}{$A$}}
\rput(8,1){\rnode{E}{$0$}}
\rput(0,0){\rnode{F}{$0$}}
\rput(2,0){\rnode{G}{$B'$}}
\rput(4,0){\rnode{H}{$C'$}}
\rput(6,0){\rnode{I}{$A$}}
\rput(8,0){\rnode{J}{$0$}}
\psset{arrows=->,nodesep=3pt,shortput=tablr,linewidth=0.1pt}
\ncline{A}{B}
\ncline{B}{C} 
\ncline{C}{D}^{$\pi$}
\ncline{D}{E}
\ncline{F}{G}
\ncline{G}{H}
\ncline{H}{I}
\ncline{I}{J}
\ncline{B}{G}<{$g$}
\ncline{C}{H}<{$G$}
\ncline{D}{I}<{id} 
\end{pspicture}
\end{center}
~\\
Note that such a $G$ does not need to be unique. By diagram chasing, we see that $G':C\to C'$ satisfies the same property as $G$ if and only if it can be written $G'=G+\Delta\circ \pi$ for some
$\Delta \in \text{Hom}(A,B')$.\\
~\\
Similarly, if $C\in \text{Ext}(A,B)$, and $f:A'\to A$, $f^*(C)$ is the unique element $C'\in \text{Ext}(A',B)$ such that there is some $F:C'\to C$ 
such that the following diagram commutes:\\
~\\
\begin{center}
\begin{pspicture}(8,1)
\rput(0,1){\rnode{A}{$0$}}
\rput(2,1){\rnode{B}{$B$}}
\rput(4,1){\rnode{C}{$C$}}
\rput(6,1){\rnode{D}{$A$}}
\rput(8,1){\rnode{E}{$0$}}
\rput(0,0){\rnode{F}{$0$}}
\rput(2,0){\rnode{G}{$B$}}
\rput(4,0){\rnode{H}{$C'$}}
\rput(6,0){\rnode{I}{$A'$}}
\rput(8,0){\rnode{J}{$0$}}
\psset{arrows=->,nodesep=3pt,shortput=tablr,linewidth=0.1pt}
\ncline{A}{B}
\ncline{B}{C} 
\ncline{C}{D}^{$\pi$}
\ncline{D}{E}
\ncline{F}{G}
\ncline{G}{H}
\ncline{H}{I}
\ncline{I}{J}
\ncline{G}{B}<{id}
\ncline{H}{C}<{$F$}
\ncline{I}{D}<{$f$} 
\end{pspicture}
\end{center}
~\\ 
As before, $F':C'\to C$ satisfies the same property as $F$ if and only if $F'=F+\Delta \circ \pi$ for some $\Delta \in \text{Hom}(A',B)$.\\
We can give an explicit description of $f^*(C)$: it is (isomorphic to) $C\times_A A'$, viewed as an extension of $A'$ via the second projection, and with map to $C$ given by the first projection.\\
~\\
An important result is Prop. 2 of Chap. 7 in \cite{Serrebook}:
An exact sequence $ 0 \to A_1 \stackrel{f}{\to} A_2 \stackrel{g}{\to} A_3 \to 0$ and an algebraic group $H$ induce an exact sequence (*):
$$
0 \to \text{Hom}(A_3,H) \stackrel{\cdot \circ g}{\to} \text{Hom}(A_2,H) \stackrel{\cdot \circ f}{\to} \text{Hom}(A_1,H) \stackrel{d}{\to}  \text{Ext}(A_3,H) \stackrel{g^*}{\to} \text{Ext}(A_2,H) \stackrel{f^*}{\to} \text{Ext}(A_1,H)
$$ 
where $d(\phi)=\phi_*(A_2)\in \text{Ext}(A_3,H)$ for $\phi \in \text{Hom}(A_1,H)$.\\
~\\
Note that in the following situation
~\\
\begin{center}
\begin{pspicture}(8,1)
\rput(0,1){\rnode{A}{$0$}}
\rput(2,1){\rnode{B}{$W_A$}}
\rput(4,1){\rnode{C}{$\tilde A$}}
\rput(6,1){\rnode{D}{$A$}}
\rput(8,1){\rnode{E}{$0$}}
\rput(0,0){\rnode{F}{$0$}}
\rput(2,0){\rnode{G}{$W$}}
\rput(4,0){\rnode{H}{$G$}}
\rput(6,0){\rnode{I}{$A$}}
\rput(8,0){\rnode{J}{$0$}}
\psset{arrows=->,nodesep=3pt,shortput=tablr,linewidth=0.1pt}
\ncline{A}{B}
\ncline{B}{C} 
\ncline{C}{D}^{$\pi_A$}
\ncline{D}{E}
\ncline{F}{G}
\ncline{G}{H}
\ncline{H}{I}
\ncline{I}{J}
\ncline{C}{H}<{$F$}
\ncline{D}{I}<{id} 
\end{pspicture}
\end{center}
~\\
where $A$ is a semiabelian variety, $G$ an extension of $A$ by a vector group $W$, and $F$ given by the universal property, $F$ must map the unipotent part $W_A$ of $\tilde A$ into the unipotent part $W$ of $G$. Hence by definition, the restriction $F_W:W_A \to W$ is such that 
$G=(F_{W})_*(\tilde A)$. Furthermore, since $\text{Hom}(A,W)=0$, $F_W$ completely determines $F$. Hence finding $F$ as in the universal property is equivalent to finding the unique $f:W_A\to W$ such that $f_*(\tilde A)=G$.\\
We will now use this characterization 
in order to build $\tilde f:\tilde A \to \tilde B$ for $f:A\to B$ an homomorphism of semiabelian varieties. For such an $f$, define $Tf$ as the unique $Tf:W_A\to W_B$ such that $(Tf)_*(\tilde A)=f^*(\tilde B)$. Because of the definitions, we get homomorphisms
$G:\tilde A \to (Tf)_*(\tilde A)$ and $F:f^*(\tilde B) \to \tilde B$ making the following diagram commutative
~\\
\begin{center}
\begin{pspicture}(12,2)
\rput(0,2){\rnode{A}{$0$}}
\rput(2,2){\rnode{B}{$W_B$}}
\rput(6,2){\rnode{C}{$\tilde B$}}
\rput(10,2){\rnode{D}{$B$}}
\rput(12,2){\rnode{E}{$0$}}
\rput(0,1){\rnode{F}{$0$}}
\rput(2,1){\rnode{G}{$W_B$}}
\rput(6,1){\rnode{H}{$f^*(\tilde B)=(Tf)_*(\tilde A)$}}
\rput(10,1){\rnode{I}{$A$}}
\rput(12,1){\rnode{J}{$0$}}
\rput(0,0){\rnode{K}{$0$}}
\rput(2,0){\rnode{L}{$W_A$}}
\rput(6,0){\rnode{M}{$\tilde A$}}
\rput(10,0){\rnode{N}{$A$}}
\rput(12,0){\rnode{O}{$0$}}
\psset{arrows=->,nodesep=3pt,shortput=tablr,linewidth=0.1pt}
\ncline{A}{B}
\ncline{B}{C} 
\ncline{C}{D}^{$\pi_B$}
\ncline{D}{E}
\ncline{F}{G}
\ncline{G}{H}
\ncline{H}{I}
\ncline{I}{J}
\ncline{G}{B}<{id}
\ncline{H}{C}<{$F$}
\ncline{I}{D}<{$f$} 
\ncline{K}{L}
\ncline{L}{M} 
\ncline{M}{N}_{$\pi_A$}
\ncline{N}{O}
\ncline{L}{G}<{$Tf$}
\ncline{M}{H}<{$G$}
\ncline{N}{I}<{id}
\end{pspicture}
\end{center}
~\\   
Now we define $\tilde f=F\circ G$, it makes the following diagram commutative
~\\
\begin{center}
\begin{pspicture}(8,1)
\rput(0,1){\rnode{A}{$0$}}
\rput(2,1){\rnode{B}{$W_B$}}
\rput(4,1){\rnode{C}{$\tilde B$}}
\rput(6,1){\rnode{D}{$B$}}
\rput(8,1){\rnode{E}{$0$}}
\rput(0,0){\rnode{F}{$0$}}
\rput(2,0){\rnode{G}{$W_B$}}
\rput(4,0){\rnode{H}{$\tilde A$}}
\rput(6,0){\rnode{I}{$A$}}
\rput(8,0){\rnode{J}{$0$}}
\psset{arrows=->,nodesep=3pt,shortput=tablr,linewidth=0.1pt}
\ncline{A}{B}
\ncline{B}{C} 
\ncline{C}{D}^{$\pi_B$}
\ncline{D}{E}
\ncline{F}{G}
\ncline{G}{H}
\ncline{H}{I}_{$\pi_A$}
\ncline{I}{J}
\ncline{G}{B}<{$Tf$}
\ncline{H}{C}<{$\tilde f$}
\ncline{I}{D}<{$f$} 
\end{pspicture}
\end{center}
~\\ 
and it is the unique such (once again using that $\text{Hom}(A,W_B)=0$).\\
With these characterizations, it is easy to show that for homomorphisms of semiabelian varieties $A \stackrel{f}{\to} B \stackrel{g}{\to} C$, 
$\tilde{gf}=\tilde g \tilde f$: the calculation $(gf)^*(\tilde C)=f^*g^*(\tilde C)=f^*(Tg)_*(\tilde B)=(Tg)_*f^*(\tilde B)=(Tg)_*(Tf)_*(\tilde A)=(TgTf)_*(\tilde A)$ shows that $T(gf)=TgTf$, and the result follows (the basic results that we use here about $f^*$ and $g_*$ can be found in \cite{Rosenlicht2} or \cite{Serrebook}).\\
~\\
Now we prove exactness.\\
We will use the natural identification of $W_A$ with the dual of $\text{Ext}(A,\Ga)$. More precisely, if $A$ is an abelian variety, the map
$$
\begin{matrix}
\text{Hom}(W_A,\Ga) & \to & \text{Ext}(A,\Ga) \cr
\phi & \mapsto & \phi_* \tilde A \cr
\end{matrix}
$$
is an isomorphism (see \cite{Rosenlicht2}, Prop. 11).\\
The same result is valid for a semiabelian variety $ 0 \to T \to S \stackrel{f}{\to} A \to 0$ instead of $A$. Indeed, since $\text{Hom}(T,\Ga)=\text{Ext}(T,\Ga)=0$, it follows from the exact sequence (*) that $f^*:\text{Ext}(A,\Ga)\to \text{Ext}(S,\Ga)$ is an isomorphism. But by construction, $\tilde S=f^* \tilde A \in \text{Ext}(S,W_A)$, and for $\phi\in \text{Hom}(W_A,\Ga)$, 
$\phi_* \tilde S=\phi_* f^* \tilde A=f^* \phi_* \tilde A$, hence the result comes from the case of abelian varieties.\\

\begin{claim}For $f:A\to B$ an homomorphism of semiabelian varieties, the following diagram commutes:
~\\
\begin{center}
\begin{pspicture}(4,2)
\rput(0,2){\rnode{A}{Hom$(W_A,\Ga)$}}
\rput(4,2){\rnode{B}{Ext$(A,\Ga)$}}
\rput(0,0){\rnode{C}{Hom$(W_B,\Ga)$}}
\rput(4,0){\rnode{D}{Ext$(B,\Ga)$}}
\psset{arrows=->,nodesep=3pt,shortput=tablr,linewidth=0.1pt}
\ncline{A}{B}^{$\simeq$}
\ncline{C}{D}^{$\simeq$} 
\ncline{C}{A}<{$\cdot \circ Tf$}
\ncline{D}{B}>{$f^*$}
\end{pspicture}
\end{center}
~\\     
\end{claim}
Indeed, for $\phi \in \text{Hom}(W_B,\Ga)$, $(\phi \circ Tf)_*\tilde A=\phi_*(Tf)_*\tilde A=\phi_*f^*\tilde B=f^*\phi_*\tilde B$.\\
~\\
Now we consider an exact sequence of semiabelian varieties
$$
0 \to A \stackrel{f}{\to} B \stackrel{g}{\to} C \to 0.
$$
\begin{claim}The induced sequence is exact :
$$
0 \to \text{Ext}(C,\Ga) \stackrel{g^*}{\to} \text{Ext}(B,\Ga) \stackrel{f^*}{\to} \text{Ext}(A,\Ga) \to 0.
$$
\end{claim}
We use the exact sequence (*) and the fact that $\text{Hom}(A,\Ga)=0$ to get the exactness on the left and on the middle. For the surjectivity, we just have to use dimensions and connectedness of these groups, since the dimension of
$\text{Ext}(A,\Ga)$ equals the dimension of the abelian part of $A$.
\begin{proposition}
The induced sequence is exact :
$$
0 \to \tilde A \stackrel{\tilde f}{\to} \tilde B \stackrel{\tilde g}{\to} \tilde C \to 0.
$$
\end{proposition}
\begin{proof}
In the following commutative diagram
~\\
\begin{center}
\begin{pspicture}(8,4)
\rput(0,3){\rnode{A}{$0$}}
\rput(2,3){\rnode{B}{$W_A$}}
\rput(4,3){\rnode{C}{$W_B$}}
\rput(6,3){\rnode{D}{$W_C$}}
\rput(8,3){\rnode{E}{$0$}}
\rput(0,2){\rnode{F}{$0$}}
\rput(2,2){\rnode{G}{$\tilde A$}}
\rput(4,2){\rnode{H}{$\tilde B$}}
\rput(6,2){\rnode{I}{$\tilde C$}}
\rput(8,2){\rnode{J}{$0$}}
\rput(0,1){\rnode{K}{$0$}}
\rput(2,1){\rnode{L}{$A$}}
\rput(4,1){\rnode{M}{$B$}}
\rput(6,1){\rnode{N}{$C$}}
\rput(8,1){\rnode{O}{$0$}}
\rput(2,0){\rnode{P}{$0$}}
\rput(4,0){\rnode{Q}{$0$}}
\rput(6,0){\rnode{R}{$0$}}
\rput(2,4){\rnode{S}{$0$}}
\rput(4,4){\rnode{T}{$0$}}
\rput(6,4){\rnode{U}{$0$}}
\psset{arrows=->,nodesep=3pt,shortput=tablr,linewidth=0.1pt}
\ncline{A}{B}
\ncline{B}{C}^{$Tf$}
\ncline{C}{D}^{$Tg$}
\ncline{D}{E}
\ncline{F}{G}
\ncline{G}{H}^{$\tilde f$}
\ncline{H}{I}^{$\tilde g$}
\ncline{I}{J}
\ncline{B}{G}
\ncline{C}{H}
\ncline{D}{I} 
\ncline{K}{L}
\ncline{L}{M}^{$f$} 
\ncline{M}{N}^{$g$}
\ncline{N}{O}
\ncline{G}{L}
\ncline{H}{M}
\ncline{I}{N}
\ncline{L}{P}
\ncline{M}{Q}
\ncline{N}{R}
\ncline{S}{B}
\ncline{T}{C}
\ncline{U}{D}
\end{pspicture}
\end{center}
~\\
the columns and the bottom row are exact. The top row is exact by the two claims and duality. It follows that the middle row is exact.  
\end{proof}

\bigskip

\noindent F. Benoist,  franck.benoist@math.u-psud.fr\\
Univ. Paris-Sud \\
Department of Mathematics, Bat. 425\\
F-91405 Orsay Cedex, France.

\medskip 

\noindent E. Bouscaren,   elisabeth.bouscaren@math.u-psud.fr\\ 
CNRS - Univ. Paris-Sud \\
Department of Mathematics, Bat. 425\\
F-91405 Orsay Cedex, France.

\medskip 

\noindent A. Pillay,  apillay@nd.edu\\
Department of Mathematics\\
University of Notre Dame\\
281 Hurley Hall\\
Notre Dame, IN 46556, USA.


\end{document}